\pdfoutput=1
\documentclass[11pt]{article}

\usepackage[utf8]{inputenc}
\usepackage[T1]{fontenc}

\usepackage{amsmath,amssymb,amsfonts,amsthm,mathrsfs,braket}

\usepackage{verbatim}

\usepackage{graphicx}

\usepackage{color}

\usepackage{pgfplots}
\usetikzlibrary{calc}

\usepackage{leftidx}

\usepackage{epstopdf}
\newtheorem{rem}{Remark}

\title{Operator learning approach for the limited view problem\\ in photoacoustic tomography
\thanks{
Department of Mathematics, University of Innsbruck, Technikerstra{\ss}e 13, A-6020 Innsbruck,
Austria, E-mail: {\tt florian.dreier@student.uibk.ac.at, sergiy.pereverzyev@uibk.ac.at,
markus.haltmeier@uibk.ac.at}
}}

\date{\today}

\author{Florian Dreier, Sergiy Pereverzyev Jr.\thanks{
from March 1, 2018 at the Department of Neuroradiology,
Medical University of Innsbruck, Anichstraße 35, A-6020 Innsbruck, Austria,
E-mail: {\tt sergiy.pereverzyev@i-med.ac.at}
}
and Markus Haltmeier}

\newcommand{\rtb}{\color{black}}

\newcommand{\rtba}{\color{black}}

\newcommand{\rte}{\color{black}}

\newcommand{\skpln}{\vspace*{\baselineskip}}
\newcommand{\mc}{\centering}
\newcommand{\ds}{\displaystyle}

\newcommand{\kl}[1]{\left(#1\right)}

\newcommand{\abs}[1]{\left\vert#1\right\vert}
\newcommand{\inner}[1]{\left\langle#1\right\rangle}
\newcommand{\norm}[1]{\left\| #1 \right\|}

\newcommand{\Om}{\Omega}

\newcommand{\Gm}{\Gamma}

\newcommand{\R}{{\mathbb R}}
\newcommand{\N}{{\mathbb N}}

\newcommand{\Dbb}{{\mathbb D}}

\newcommand{\tm}[1]{\mathrm{#1}}

\newcommand{\Pm}{\mathbf{P}}

\newcommand{\cv}{\mathbf{c}}

\newcommand{\uv}{\mathbf{u}}

\newcommand{\Uv}{\mathbf{U}}

\definecolor{lightgray}{RGB}{211,211,211}
\definecolor{lightgreen}{RGB}{144,238,144}
\definecolor{lightorange}{RGB}{255,186,102}
\definecolor{darkorange}{RGB}{255,140,0}

\newcommand{\pOm}{\partial\Omega}

\newcommand{\Ecl}{\mathcal{E}}

\newcommand{\Gcl}{\mathcal{G}}
\newcommand{\Zcl}{\mathcal{Z}}
\newcommand{\Pcl}{\mathcal{P}}

\newcommand{\Ra}{\mathbb{Y}}
\newcommand{\Rb}{\mathbb{Y}_1}
\newcommand{\Rc}{\mathbb{Y}_2}

\newcommand{\Qcl}{\mathcal{Q}}

\newcommand{\Lcl}{\mathcal{L}}

\newcommand{\Ucl}{\mathcal{U}}

\newcommand{\Ufr}{\mathfrak{U}}

\newcommand{\Uo}{\mathcal{U}}
\newcommand{\Ao}{\mathcal{A}}
\newcommand{\hAo}{\hat\Ao_n}

\newcommand{\fze}{\hat f_{ 0 }}
\newcommand{\fle}{\hat f_{ n }}

\newcommand{\uze}{\hat u_{ 0 }}
\newcommand{\ule}{\hat u_{ n }}

\newcommand{\Eze}{\Ecl_{ 0 }}
\newcommand{\Ele}{\Ecl_{ n }}

\newcommand{\kdb}{\kappa_d}
\newcommand{\pdt}{\partial_t}
\newcommand{\Dclt}{\mathcal{D}_t}
\newcommand{\df}{\tm{d}}

\DeclareMathOperator{\supp}{supp}
\DeclareMathOperator{\conv}{conv}

\theoremstyle{plain}
\newtheorem{thm}{Theorem}
\newtheorem{cor}{Corollary}

\theoremstyle{remark}

\usepackage{geometry}
\allowdisplaybreaks

\begin{document}
	\maketitle

\begin{abstract}

\noindent
In photoacoustic tomography, one is interested to recover the initial pressure distribution inside a tissue from the corresponding
measurements of the induced acoustic wave
on the boundary of a region enclosing the tissue.
In the limited view problem, the wave boundary measurements are given
on the part of the boundary,
whereas in the full view problem, the measurements are known on the whole boundary.
For the full view problem, there exist various fast and robust reconstruction methods. These methods give severe reconstruction
artifacts when they are applied directly to the limited view data.
One approach for reducing such artefacts is trying to extend the limited view data to the
whole region boundary, and then use existing reconstruction methods for the full view data.
In this paper, we propose an operator learning approach for constructing an 
operator that gives an approximate extension of the limited view data.
We consider the behavior of a reconstruction formula on the extended limited view data that is given
by our proposed approach. Approximation errors of our approach are analyzed.
\rtb
We also present numerical results with the proposed extension approach supporting our theoretical analysis.
\rte

\skpln

\skpln

\end{abstract}


\section{Introduction}

Photoacoustic tomography (PAT) is an emerging non-invasive imaging technique.
It is based on the photoacoustic effect, and it has a big potential for a successful use in biomedical studies,
including preclinical research and clinical practice.
Applications include
tumor angiogenesis monitoring, blood oxygenation mapping, functional brain imaging,
and skin melanoma detection~\cite{XuWan06,LiWan09,Bea11,XiaYaoWan14}


The principle of PAT is the following.
When short pulses of  non-ionising electromagnetic energy are delivered into a biological (semi-transparent) tissue, then  parts of the electromagnetic energy become absorbed.
The absorbed energy leads to a nonuniform thermoelastic expansion depending  on the tissue structure.
This gives rise to an initial acoustic pressure distribution, which further
is the source of an
acoustic pressure wave.
These waves are detected by a measurement device on the boundary of the tissue.
The mathematical task in PAT is to reconstruct the spatially varying initial pressure distribution using these measurements.
The values of the initial pressure distribution inside the tissue allow to make a judgment about the directly unseen structure of the tissue.
For example, whether there are some abnormal formations inside the investigated tissue, such as a tumor.


Consider the part of the boundary of
a region enclosing the
tissue
where the wave measurements are available. This part is called observation boundary.
If the tissue is fully enclosed by the observation boundary, then one speaks about the full view problem. Otherwise, if some part of the
tissue boundary is not accessible, then one has the so-called limited view problem (LVP). The LVP frequently arises in practice,
for example in breast imaging (see, e.g., \cite{XuWanAmbKuc04,KucKun11}).

The LVP can be approached using iterative reconstruction algorithms
(see, e.g., \cite{PalViaPJ02,PNH07,Her09,YaoJia11,HuaWanNWA13,HalNgu17,RosNtzRaz13}).
Although these algorithms can provide accurate reconstruction,
they are computationally expensive and time consuming. Approaches for the full view problem,
such as time reversal~\cite{BurMatHalPal07,HriKucNgu08},
Fourier domain algorithms~\cite{Hal09,Kun07b,XuXuWan02},
explicit reconstruction formulas~\cite{FinPatRak04,FinHalRak07,Kun07,Kun11,Lin14},
are faster than iterative reconstructions and additionally are robust and accurate. However,
when they are directly applied on the limited view data, then one obtains severe reconstruction artifacts.

And so, an idea appears to try to extend the limited view data to the whole boundary, and then use efficient algorithms
for the full view data on the extended data to obtain a reconstruction of the initial pressure.
Knowing characterizations of the range of the forward operator, which maps the initial pressure distribution to the wave data on the
whole boundary of the tissue, may be used for this purpose
(see, e.g., \cite{AmbKuc06,FinRak06,AgrFinKuc09,KucKun11}
and the references therein).
This knowledge is expressed with so-called range conditions. In~\cite{Pat04,Pat09}, some of these conditions,
the so-called moment conditions, were realized for the extension of the limited view data.

The data extension process based on the moment conditions is unstable,
and therefore, mostly low frequencies of the limited view wave data
can be extended. This instability is connected with the following issue.
The observation boundary defines a so-called detection region,
\rtb
which, for typical measurement configurations, is
the convex hull of the observation boundary~\cite{KucKun08}.
\rte
It is known (see, e.g., \cite{KucKun08,SteUhl09,KucKun11}) that if the support of the initial pressure is contained in this detection region,
then a stable recovery of the initial pressure from the limited view wave data is theoretically possible.
However, the data extension process based on the moment conditions
does not use information about the support of the initial
pressure, and so, it does not employ advantages of the possible stable recovery.

In this paper, we propose a stable method for the extension of the limited view wave data that uses advantages of the
mentioned possible stable recovery. Our method is based on the observation that in the case of the stable
recovery, there exists a continuous data extension operator that maps the limited view wave data to the unknown wave data
on the unobservable part of the boundary.
\rtb
We formally define this operator in Section~\ref{s:EO}.
\rte
However, this operator is not explicitly known.
In our method, we therefore propose to construct an approximate data extension operator using an operator
learning approach that is \rtb inspired \rte by the methods of the statistical learning theory
(see, e.g., \cite{HasTibFri09}).
\rtb
We suggest an operator learning procedure that uses the projection on the linear subspace defined by the training inputs.
\rte

Having an approximately extended limited view wave data, one can employ reconstruction methods for the full
view wave data, such as time reversal or methods based on the explicit inversion formulas.
As an example, we consider an explicit reconstruction formula for that purpose. We demonstrate that the resulting
reconstruction algorithm corrects
most
limited view reconstruction artifacts, while the computational time remains to be low.
The involved steps in the proposed reconstruction approach are illustrated in Figure~\ref{Fig:OLa}.

\begin{figure}[t]
\begin{center}
    \includegraphics[width=15cm]{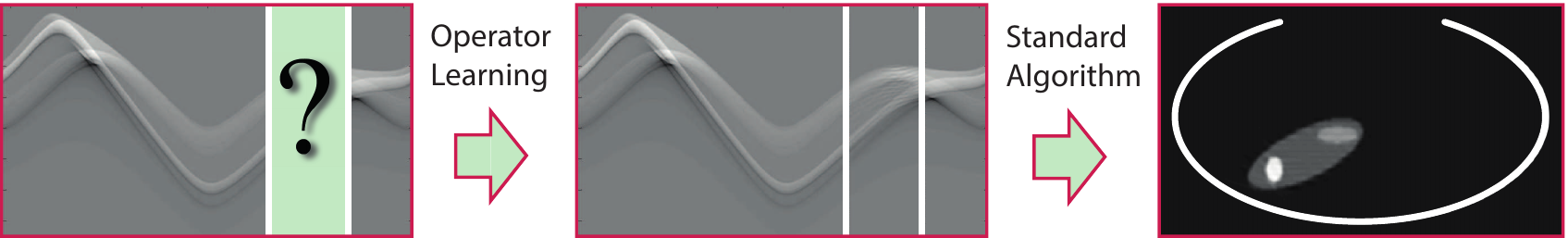}
\end{center}
\caption{
Illustration of the proposed approach for limited view PAT.
In the first step, we extend the limited view data to the whole boundary via operator learning.
In the second step, we apply a standard direct PAT reconstruction algorithm to the completed data.
\label{Fig:OLa}}
\end{figure}

The rest of the paper is organized as follows. In section~\ref{s:Mat}, we present a mathematical background
for PAT, give the used explicit reconstruction formula, and discuss the LVP. Our operator learning approach
to the extension of the limited view wave data is given in section~\ref{s:OL}. In section~\ref{s:EA}, we analyze
the approximation errors of our approach. We look at the approximation errors for the unknown wave data
and for the corresponding reconstructions obtained by explicit reconstruction formulas.
We present the numerical results in section~\ref{s:NR}. Finally, we finish the paper with conclusion and outlook
in section~\ref{s:Con}.


	
\section{Mathematics of PAT}\label{s:Mat}

Let $\Om\subseteq\R^d$ be a bounded domain with a smooth boundary $\partial\Om$, where $d\geq 2$ denotes the spatial dimension.
Further, let $C_c^\infty\kl{\Om}$ be the set of all smooth functions
$f \colon \R^d\to\R$ that are compactly supported in $\Om$.
In PAT, one is interested to recover
an unknown function $f \in C_c^\infty\kl{\Om}$ from the solution of the wave equation given on the boundary of $\Om$.
Let us mathematically specify this reconstruction problem.

\subsection{Reconstruction problem}

Let $\Ufr f  \colon \R^d\times (0,\infty) \to \R$ denote the solution of the following initial value problem for the wave equation:
\begin{equation}\label{ivp_we}
\left\{
\begin{aligned}
    (\partial_t^2 - \Delta_x )\, u(x,t) =& 0
    && \text{ for }  (x,t)\in\R^d\times(0,\infty),  \\
    u(x,0)=& f(x) && \text{ for } x\in\R^d, \\
    (\partial_t u)(x,0)=& 0
    && \text{ for }  x\in\R^d \,.
\end{aligned}
\right.
\end{equation}
Here $\partial_t$ denotes differentiation with respect to the second variable $t$,
and $\Delta_x$ is the Laplacian with respect to $x$.
Then the reconstruction problem in PAT consists in recovering the unknown function $f\in C^\infty_c(\Om)$ from
the corresponding wave boundary data
\begin{equation}\label{pr2}
  u(x,t) = \kl{\Ufr f} (x,t)
  \quad
  \text{ for }  (x,t)\in \Gm_1\times (0,\infty)  \,,
\end{equation}
where $\Gm_1\subseteq\pOm $.
If $\Gm_1 = \pOm$, then~\eqref{pr2} is called full view problem; otherwise,
if $\Gm_1\subsetneq \pOm $, we have the limited view problem (LVP). In this paper, we are particularly interested in the
limited view case, which we consider in some detail in subsection~\ref{s:LVP}.


Let us denote the unobservable part of the boundary as $\Gm_2 := \pOm\setminus \Gm_1$. We define also the following restrictions
of $\Ufr f$:
\begin{equation}\label{U1U2}
  \Ucl f := \Ufr f | _{ \pOm\times (0,\infty) },\;
  \Ucl_1 f := \Ufr f | _{ \Gm_1\times (0,\infty) },\;
  \Ucl_2 f := \Ufr f | _{ \Gm_2\times (0,\infty) }.
\end{equation}
%


Let us note that in practice, the reconstruction problem~\eqref{pr2} arises in PAT in spatial dimensions
two and three. The three dimensional problem appears when the so-called
point-like detectors are used (see, for example,  \cite{XuWan06, KucKun08, FinRak09}).
When one uses linear or circular integrating detectors, then the reconstruction problem~\eqref{pr2}
is considered in two spatial dimensions (see \cite{BBG07,GruJBO10,PalNusHalBur07,ZanSchHal09b}).


\subsection{Explicit inversion formula}

The reconstruction problem~\eqref{pr2} can be approached by various solution techniques.
Among these techniques, the derivation of the explicit inversion formulas
of the so-called back-projection type is particularly appealing.
A numerical realization of these formulas typically gives reconstruction algorithms
that are accurate and robust, and at the same time  are faster than iterative approaches.

An inversion formula consists of an explicitly given operator
$\Gcl_d$ that recovers the function $f$ from the data $u$.
Such formulas are currently known only for special domains
and only for the full view data, i.e. $u$ must be given for all $x\in\pOm$.
In this paper, we consider the formula that first has been derived
in~\cite{XuWan05,Kun07,BBG07}.
In addition to the data $u$, the formula $\Gcl_d$ also depends on the boundary $\pOm$ of the domain
$\Om\subsetneq \R^d$ and on the reconstruction point $x_0\in\Om$.
The structure of the formula further depends on whether the spatial dimension $d$ is even or odd.

If $d\geq 2$ is an even integer, then
\begin{equation}\label{rec-U1}
\Gcl_d(\pOm,u,x_0):=
  \kdb  \int_{\partial\Om} \inner{\nu_x,x_0-x}
    \int_{ \abs{x_0-x} }^\infty
    \frac{ \kl{ \pdt \Dclt^{ (d-2)/2 } t^{-1}  u } (x,t) }
    { \sqrt{ t^2- \abs{x_0-x}^2 } }
    \,\df t\,
    \df s(x) \,.
\end{equation}
Here
$\ds\kdb:= (-1)^{(d-2)/2}/\pi^{d/2}$ is a constant,
$\nu_x$ denotes the outward pointing unit normal to $\pOm$,
and $\Dclt:=(2t)^{-1}\partial_t$ is the differentiation operator with respect to $t^2$.
Further,  $\inner{\,\cdot\, , \,\cdot\,}$ and $\abs{\,\cdot\,}$ denote the standard inner product and the
corresponding Euclidian norm on $\R^d$, respectively.

In the case of odd dimension $d\geq 3$, the formula $\Gcl_d$ is defined as follows:
\begin{equation}\label{rec-U2}
\Gcl_d(\pOm,u,x_0):=
  \kdb  \int_{\partial\Om}
    \frac{\inner{\nu_x, x_0-x}}{ \abs{x_0-x}  }
    \kl{ \pdt \Dclt^{ (d-3)/2 } t^{-1}  u } (x,\abs{x_0-x})
    \,\df s(x) \,,
\end{equation}
with constant $\ds\kdb:= (-1)^{(d-3)/2}/( 2\pi^{ (d-1)/2 })$.

The  formula $\Gcl_d$ has been introduced in~\cite{XuWan05}
for dimension $d=3$, and in~\cite{BBG07} for dimension $d=2$.
In~\cite{Kun07}, it has been studied for the case when $\Om$
is a ball in arbitrary dimension.
Further, in~\cite{Nat12,Hal13,Hal14}, it has been shown that for any elliptical domain  $\Om$,  the formula
$\Gcl_d$ exactly recovers any smooth function $f$ with support in  $\Om$ from data $u = \Uo f$.
In~\cite{HalPer15}, it was shown, that the same result also holds for parabolic domains $\Om$ with $d=2$.
The formula $\Gcl_d$ in arbitrary spatial dimension $d\geq 2$ on certain quadric hypersurfaces, including the parabolic ones,
has been analyzed in~\cite{HalPer15b}.

It should be noted that the formula $\Gcl_d$ can be in fact used for any convex bounded domain $\Om$.
Then, however, the formula does not recover the function $f$
exactly, and it introduces an approximation error. The form of this error has been analyzed in~\cite{Nat12,Hal13,Hal14}.
Numerical experiments indicate
that this error is rather low for domains that can be well approximated by elliptic domains.
\rtb
This is also suggested by the microlocal analysis in~\cite{Lin14}.
\rte


The operator $\Ucl$ can be defined for functions $f\in\Lcl^2\kl{\Om_0}$,
where $\Om_0$ is an open set with $\overline{\Om_0} \subseteq \Om$.
Define the image of $\Lcl^2\kl{\Om_0}$ under the operator
$\Ucl$ as $\Ra:=\Ucl\kl{ \Lcl^2\kl{\Om_0} }$. Then it is known
(see, e.g., \cite{KucKun08,SteUhl09,KucKun11})
that $\Ra$ is a closed subspace of
$\Lcl^2(\pOm\times(0,\infty))$, and therefore, we will treat $\Ra$ as a Hilbert space with the scalar product of $\Lcl^2(\pOm\times(0,\infty))$.
Moreover, the operator $\Ucl\colon \Lcl^2\kl{\Om_0} \to\Ra $ is bounded, and it has the bounded inverse
$\Ucl^{-1}\colon \Ra \to \Lcl^2\kl{\Om_0}$.

In the following, we will work with functions $f\in\Lcl^2\kl{\Om_0}$, and we will assume that the domain $\Om$ is such that
the formula $\Gcl_d$ gives exact recovery of the function $f$ from its wave data
$u=\Uo f$, i.e. it holds that
\begin{equation}\label{exrec}
   f = \Gcl_d\, \Uo f.
\end{equation}
As we already mentioned, this is, for example, the case for circular and elliptical domains.
In such a situation, it can be shown that $\Gcl_d$ is a continuous extension of
$\Ucl^{-1}$  to $\Lcl^2(\pOm \times (0, \infty))$.


\subsection{Limited view problem}\label{s:LVP}


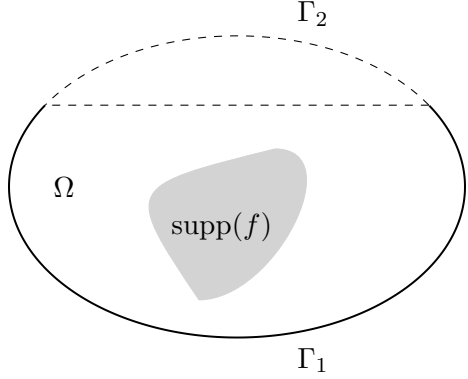
\begin{figure}[t]
		\centering
		\begin{tikzpicture}
			\draw[,domain=pi/2+1:2*pi+pi/2-1,samples=300,line width=0.75pt,variable=\x] plot({3*cos(\x r)},{2*sin(\x r)});
			\draw[dashed,domain=pi/2-1:pi/2+1,samples=300,variable=\x] plot({3*cos(\x r)},{2*sin(\x r)});
			\draw[dashed] ($3*cos((pi/2-1) r)*(1,0)+2*sin((pi/2-1) r)*(0,1)$)--($3*cos((pi/2+1) r)*(1,0)+2*sin((pi/2+1) r)*(0,1)$);
			\draw(1,2) node[anchor=south]{$\Gamma_2$};
			\draw(1,-2) node[anchor=north]{$\Gamma_1$};
			\draw(-2,0) node[anchor=east]{$\Omega$};
			\draw[fill=lightgray,draw=lightgray] (-0.5,-1.5) .. controls (0.5,-1.5) and (1.5,0.5) .. (0.5,0.5) .. controls (-1.5,0) .. (-0.5,-1.5);
			\draw(-0.2,-0.2) node[anchor=north]{$\mathrm{supp}(f)$};
		\end{tikzpicture}
		\caption{  Setting of LVP.  }
		\label{fig:edaten}
\end{figure}


In practice, the wave data $u$ is frequently given on a subset $\Gm_1$ of the boundary $\pOm$ (Figure~\ref{fig:edaten}).
This subset $\Gm_1$, called observation boundary, defines the so-called detection region
$\Dbb\kl{ \Gm_1 }$ (see, e.g., \cite{PNH07,KucKun08}).
If $\overline{\supp(f)} \subsetneq \Dbb\kl{ \Gm_1 } $,
then the function $f$ in~\eqref{pr2} can be stably recovered from data on $\Gm_1$.
The detection region $\Dbb\kl{ \Gm_1 }$ contains points $x$ such that any line going through $x$ intersects $\Gm_1$.
For example, if $\Gm_1$ is a spherical or elliptical cap, then
$\Dbb\kl{ \Gm_1 } = \conv\kl{ \Gm_1 }$.

Let us mathematically specify the stable recovery of $f$. Let $\Om_1$ be an open set with
$\overline{\Om_1} \subsetneq \Dbb\kl{ \Gm_1 } $.
\rtba
The stable recovery holds for $f \in \Lcl^2\kl{\Om_1} $, and it is formulated in the following theorem.
Note that the space $ \Lcl^2\kl{\Om_1} $ is identified 
with the set of all functions in  $\Lcl^2(\R^d)$ that vanish outside  of  $\overline{\Om_1}$.

\begin{thm}\label{Th_sr}
The operator $\Ucl_1 \colon \Lcl^2\kl{\Om_1} \to \Lcl^2\kl{ \Gm_1 \times (0, \infty)}$ is well defined 
and bounded. Moreover,  it has  bounded inverse $\Ucl_1^{-1}\colon \Rb \to \Lcl^2\kl{\Om_1}$, where  
$\Rb := \Ucl_1 (\Lcl^2\kl{\Om_1}) \subseteq \Lcl^2 (\Gm_1 \times (0, \infty))$ denotes the range of $\Ucl_1$.
In particular, $\Rb$ is closed.
\end{thm}


\begin{proof}
It is sufficient to show the two-side estimate
\begin{equation} \label{eq:two}
\forall f \in C_c^\infty (\Om_1) \colon \quad 
a \norm{\Ucl_1 f}_{\Lcl^2} \leq \norm{f}_{\Lcl^2}  \leq b \norm{\Ucl_1 f}_{\Lcl^2} \,,
\end{equation}
for some constants $a, b \in (0, \infty)$. The claims then follow by continuous extension. 

To show the left hand estimate, we decompose 
$ \Ucl_1 f =  \chi_{[0,T]}  \Ucl_1 f   + \chi_{(T, \infty)} \Ucl_1 f$,
where $T$ is larger than the diameter of $\Om$.
Since the operator  $\Ucl_1$ is the sum of two Fourier integral operators  of order zero (see \cite{HalNgu17}),
we have $ \lVert \chi_{[0,T]} \Ucl_1 f\rVert_{\Lcl^2} \leq  c_1 \lVert f\rVert_{\Lcl^2}$ for some constant $c_1$.
Moreover, the explicit formulas for $\Ucl_1 f$ (see, e.g, \cite{CouHil62}) imply also that
$ \lVert \chi_{(T, \infty)} \Ucl_1 f \rVert_{\Lcl^2} \leq c_1 \lVert f\rVert _{\Lcl^2}$, which gives the 
left hand side estimate in \eqref{eq:two}.

The right hand side estimate can be found in~\cite[Theorem~3.4]{HalNgu17}. The required visibility condition
is satisfied for 
$f \in \Lcl^2\kl{\Om_1} $.
\end{proof}


It is worth to mention that despite the boundedness of $\Ucl^{-1}_1$, no theoretically exact direct solution methods
are available.
Let us note that if the condition $\overline{\Om_1} \subsetneq \Dbb\kl{ \Gm_1 } $ is not satisfied, then
the visibility condition in~\cite[Theorem~3.4]{HalNgu17} is also not valid, and the
inverse of the operator $\Ucl_1$ is severely ill-posed (see, e.g., \cite{HalNgu17,SteUhl09,KucKun11}).


\rte


Denote $\Rc := \Lcl^2\kl{ \Gm_2\times (0,\infty) }$.
From the boundness of the operator $\Ucl\colon \Lcl^2\kl{ \Om_0 } \to \Ra $, we can deduce the boundness of the operator
$\Ucl_2\colon \Lcl^2\kl{ \Om_1 } \to \Rc $. We will use this for the data extension operator below.

\rtba
Recall that in order to give the exact reconstruction, the formula $\Gcl_d$ requires the full view wave data $u$,
which is given for all $x\in\pOm$ (see~\eqref{exrec}).
\rte
In spite of the \rtba above discussed \rte stable
recoverability of $f\in \Lcl^2\kl{ \Om_1 }$ from equation~\eqref{pr2},
the use of formula $\Gcl_d$ on the \rtba limited view \rte data $u$ given on $\Gm_1\subsetneq\pOm$
leads to serious artifacts in the reconstruction;
see, e.g., \cite{HalPer15}, where the numerical results of the application of $\Gcl_2$ on finite parabolas are presented.
The reconstruction artefacts in the case of the limited view data are also discussed
in~\cite{XuWanAmbKuc04,FriQui13,SteUhl13,BarFriNgu15,FriQui15,Ngu15}.

At the same time,
the use of formula $\Gcl_d$ for reconstructing function $f$ can be attractive from various points of view. For example, as we already pointed out,
the reconstruction using a numerical realization of $\Gcl_d$ is
\rtb
faster than iterative reconstruction algorithms.
\rte
Another point may be connected with the nature of the software development.
Namely, having already a tested and trusted computer code of the numerical realization of formula $\Gcl_d$, it could be
\rtb
tempting
\rte
to develop its extensions for the LVP.

An extension of the limited view data $u$ from the observable part of the boundary $\Gm_1\subsetneq\pOm$ to the whole
boundary $\pOm$ may give a possibility to improve the reconstruction quality of the formula $\Gcl_d$. In this paper,
we propose to realize this extension using the operator learning approach, which we consider in the next section.

	
\section{Data extension using operator learning approach}\label{s:OL}

The extension of the limited view data to the whole boundary can be in principle done by the extension operator that we define
in the next subsection. This operator is however not explicitly known, and we propose an operator learning approach to
construct its approximation in subsection~\ref{s:LEO}. In subsection~\ref{s:CLA}, we discuss computational aspects of the
proposed learned approximation of the extension operator.

\subsection{Extension operator}\label{s:EO}

\rtb
Let us recall that $\Gm_1\subsetneq \pOm $ is the observation boundary,
$\Dbb\kl{ \Gm_1 }$ is the corresponding detection region defined in Section~\ref{s:LVP},
$\Gm_2 = \pOm\setminus \Gm_1$ is the unobservable part of the boundary,
and $\Om_1$ is an open set with $\overline{\Om_1} \subsetneq \Dbb\kl{ \Gm_1 } $.
Further, let us remind that the operators $\Ucl_1$ and $\Ucl_2$ are defined in~\eqref{U1U2}.
\rte


The operator $\Ao\colon \Rb\to \Rc$ that maps functions
$\Ucl_1 f$ to functions $\Ucl_2 f$ for $f\in\Lcl^2\kl{ \Om_1 }$
realizes the extension of the limited view data $u_1=\Ucl_1 f$ to the unobservable part of the boundary $\Gm_2$.
This operator $\Ao$ can be written as $\Ao = \Uo_2 \circ \Uo_1^{-1} $.
Because of this representation
and the assumptions on $\Gamma_1$ and $\Omega_1$,
the operator $\Ao$ is a linear continuous operator
\rtba
as a superposition of linear continuous operators.
Recall that the continuity (or boundness) of the operators $\Ucl_1^{-1}$ and $\Ucl_2$ is discussed in
Section~\ref{s:LVP}.
\rte


With the introduced extension operator $\Ao$, one could extend the limited view data $u_1$ to the whole boundary $\pOm$,
and then use the formula $\Gcl_d$ on this extended data. In this way, the disadvantages of the use of the formula $\Gcl_d$
on the limited view data can be eliminated. However, the form of the operator $\Ao$ is not explicitly known.

\subsection{Proposed learned extension operator}\label{s:LEO}

In this paper, we propose to construct an operator $\hAo$ that approximates the operator $\Ao$.
The role of the parameter $n\in\N\cup\Set{0}$ is described below.
The approximate operator $\hAo$
must satisfy the following two requirements. The first requirement concerns the approximation quality: $\hAo u_1$ must be close to
$\Ao u_1$. The second requirement is related to the computational effort of the numerical evaluation of $\hAo u_1$. This evaluation
must be fast such that the evaluation of the formula $\Gcl_d$ on the extended limited view data with the help of $\hAo$ remains to
be computationally efficient.

Our construction of the approximate operator $\hAo$ is
\rtb inspired \rte
by the statistical learning approach (see, e.g., \cite{HasTibFri09}).
For $i=1,\ldots,n$, consider training functions $f_i\colon\Om_1\to\R$.
For each training function $f_i$, we can determine the corresponding
wave data
$u_{1,i}:= \Uo_1 f_i$, $u_{2,i}:= \Uo_2 f_i$.
By the definition of the extension operator $\Ao$ we have that $u_{2,i}=\Ao u_{1,i}$. In the context of statistical learning, the set
$ \Zcl := \Set{  \kl{  u_{1,i}, \Ao u_{1,i} },\;  i=1,\ldots,n  } $
is called a training set. Define for future reference the set $\Uv_{1,n}:=\Set{u_{1,i},\; i=1,\ldots,n}$.

So, how to construct (or, using the terminology of the statistical learning, how to learn) an approximation $\hAo u_1$ of $\Ao u_1$
using the training set $\Zcl$?
\rtba
It should be noted that many
statistical learning algorithms are designed for learning
a small number of
scalar-valued functions.
These algorithms are not applicable in our case because the function that we need to learn is an operator.
Recently, the development of the statistical learning methods for learning vector-valued functions
and also functions with values in function spaces, i.e. operators, has been started
(see, e.g., \cite{MicPon05,AlvRosLaw12}). For obtaining good results, these methods require
an a priori knowledge of the dependence between different components of the output vector
that is given by the function to be learned. This knowledge is not readily available in our case.
However, as we observe below, the linear structure of the extension operator $\Ao$ that we want
to learn allows to employ a projection operator for the learning.
\rte


For any $n\in\N\cup\Set{0}$, define the linear subspace
\begin{equation}\label{VN}
  V_n := \Set{\sum_{j=1}^n c_j u_{1,j},\; c_j\in\R  },\; V_0:=\Set{ 0 } \subseteq\Rb,
\end{equation}
and let $\Pcl_n\colon \Lcl^2\kl{ \Gm_1\times(0,\infty) } \to V_n$ be the orthogonal projection on $V_n$ in
$\Lcl^2\kl{ \Gm_1\times(0,\infty) }$.
Then we define the learned approximation $\hAo u_1$ as follows:
\begin{equation}\label{hA}
  \hAo u_1 := \Ao \Pcl_n u_1.
\end{equation}
%

Note that $V_n\subseteq \Rb$, and therefore, the operator composition $\Ao \Pcl_n $ is well-defined, and
$\hAo\colon \Lcl^2\kl{ \Gm_1\times (0,\infty) } \to\Rc $ is bounded.
Further, note that for all $u_1\in  \Lcl^2\kl{ \Gm_1\times (0,\infty) } $, $\hat \Ao_0 u_1 = 0\in\Rc$.

\subsection{Computation of learned approximation}\label{s:CLA}

How to compute the learned approximation $\hAo u_1$ using the training set $\Zcl$ for $n\geq 1$? First of all,
observe that since $\Pcl_n u_1\in V_n$, the projection
$\Pcl_n u_1$ has the following representation:
\begin{equation}\label{Pnrpr}
  \Pcl_n u_1 = \sum\limits_{j=1}^n c_j u_{1,j},
\end{equation}
where the coefficients $c_j\in\R$ can be determined from the conditions
$ \inner{ \Pcl_n u_1 -u_1, u_{1,i} } = 0 $ for $i=1,\ldots,n$. These conditions can be written in the form of the system of linear
equations for the coefficients $c_j$
\begin{equation}\label{sysline}
   \sum\limits_{j=1}^n c_j \inner{ u_{1,i},u_{1,j} }  =  \inner{ u_1,u_{1,i} },\;
	 i=1,\ldots,n.
\end{equation}
%

Denote the matrix corresponding to the above linear system as $\Pm_n$, i.e. the elements of $\Pm_n$ are
$\kl{ \Pm_n }_{ij} = \inner{ u_{1,i},u_{1,j} }$.
Further, denote the vector of unknowns as $\cv_n$, and the right-hand side as $\uv_n$, i.e.
$ \kl{ \cv_n }_i = c_i $ and $ \kl{ \uv_n }_i = \inner{ u_1, u_{1,i} } $.

The matrix $\Pm_n$ is the Gram matrix of the functions in $\Uv_{1,n}$, and it is invertible
if the set $\Uv_{1,n}$ is linearly independent. Since the operator $\Uo_1$ is invertible, the set $\Uv_{1,n}$ is linearly independent if the
set $\Set{ f_i,\; i=1,\ldots,n }$ is linearly independent, and for the following, we assume that this is the case.

Note that the matrix $\Pm_n$ does not depend on the limited view wave data $u_1$ that we want to extend. Therefore, the inverse
matrix $\Pm_n^{-1}$ can be precomputed once the set of the learning inputs $\Uv_{1,n}$ is given. This will make the determination
of the coefficients $c_j$ very fast.

Finally, with the coefficients $c_j$ in~\eqref{Pnrpr},
i.e. $\cv_n = \Pm_n^{-1} \uv_n$,
the approximation $\hAo u_1$ is calculated as follows:
$$
  \hAo u_1 = \Ao\Pcl_n u_1 =
  \Ao\kl{  \sum\limits_{j=1}^n c_j u_{1,j}  } = \sum\limits_{j=1}^n c_j u_{2,j}
	= \sum\limits_{j=1}^n c_j \Uo_2 f_j.
$$
%

	
\section{Approximate reconstructions and their error analysis}\label{s:EA}

For obtaining an approximate reconstruction of $f$ using
the limited view data $u_1 = \Uo_1 f$ and the formula $\Gcl_d$, we can now proceed as follows.
First, we extend the limited view data $u_1$ to the whole boundary $\pOm$ using the learned extension operator $\hAo$
in this way:
$$
   \ule(x,t) =
	\begin{cases}
		u_1(x,t) & \text{if }  x\in\Gm_1, \\
		\bigl(   \hAo u_1   \bigr)     (x,t) & \text{if }  x\in\Gm_2.
	\end{cases}
$$
And then we apply the formula $\Gcl_d$ to this extended wave data $\ule$
in order to obtain an approximate reconstruction
$\fle$:
\begin{equation}\label{hf2}
  \fle = \Gcl_d \ule.
\end{equation}

Note that $\uze$ is obtained by extending the limited view data $u_1$ to the whole boundary $\pOm$ with zero values
on $\Gm_2$. As we already discussed, the corresponding approximate reconstruction $\fze$
contains significant errors, and it is desirable
to have better reconstructions of $f$ using $u_1$.
Additionally, one may desire that the reconstruction $\fle$ improves as $n$ increases.

In the following theorem, we estimate the $\Lcl^2$-error of the approximation of $\Ao u_1$ by $\hAo u_1$ and of the approximation
of $f$ by $\fle$. From the derived estimates, we see that the above aims can be realized if the
training functions $f_i$, $i=1,\ldots,n$, are chosen appropriately.


\begin{thm}\label{Th1}

Let a set of linearly independent training functions
$\Set{f_i,\; i=1,\ldots,n } \subseteq \\ \Lcl^2\kl{ \Om_1 }   $ be given, and denote
$W_n := \Set{  \sum_{i=1}^n c_i f_i,\;  c_i\in\R  }$,
$W_0 := \Set{ 0 } \subseteq \Lcl^2\kl{ \Om_1 } $.
Define the training limited view wave data
$u_{1,i} := \Uo_1 f_i$, the corresponding linear subspace $V_n$ in~\eqref{VN}, and the learned extension
operator $\hAo$ in~\eqref{hA}. Consider a function $f\in\Lcl^2\kl{ \Om_1 }$, its limited view wave data $u_1:=\Uo_1 f$,
and its approximation $\fle$ defined in~\eqref{hf2}. Then the following $\Lcl^2$-error estimate
for the unobservable data holds:
\begin{equation}\label{ere1}
   \norm{ \Ao u_1 -\hAo u_1 } \leq \norm{ \Ao }\cdot \norm{\Uo_1}\cdot \min\limits_{g\in W_n} \norm{ f-g }.
\end{equation}
If additionally, the domain $\Om$ is such that~\eqref{exrec} holds, then we have the following $\Lcl^2$-error estimate for the reconstruction:
\begin{equation}\label{ere2}
   \norm{ f - \fle } \leq \norm{ \Gcl_d }\cdot \norm{ \Ao }\cdot \norm{\Uo_1}\cdot \min\limits_{g\in W_n} \norm{ f-g }.
\end{equation}
\end{thm}

\begin{proof}

We first prove~\eqref{ere1}. From the definition of the operator $\hAo$, we have that
%
\begin{equation}\label{AuhAu}
  \norm{ \Ao u_1 - \hAo u_1 } = \norm{ \Ao\, \Uo_1 f - \Ao\, \Pcl_n\, \Uo_1 f }\leq
	\norm{\Ao} \cdot \norm{ \Uo_1 f - \Pcl_n\, \Uo_1 f }.
\end{equation}
From the properties of the projection operators, we also have that
\begin{equation}\label{U1fPN}
  \norm{ \Uo_1 f - \Pcl_n\, \Uo_1 f } = \min\limits_{h\in V_n} \norm{ \Uo_1 f - h }.
\end{equation}

For an element $h\in V_n$, there are unique constants $c_i\in\R$, $i=1,\ldots,n$ such that
$$
  h = \sum_{i=1}^n c_i\, u_{1,i} = \sum_{i=1}^n c_i\, \Uo_1 f_i =
	\Uo_1\kl{ \sum_{i=1}^n c_i\, f_i },
$$
and therefore, there exists an element $g\in W_n$ such that $h=\Uo_1 g$. Using this fact, we can
estimate
\begin{equation}\label{U1fh}
  \min\limits_{h\in V_n} \norm{ \Uo_1 f - h } = \min\limits_{g\in W_n} \norm{ \Uo_1 f - \Uo_1 g }
	\leq \norm{ \Uo_1 } \cdot \min_{g\in W_n} \norm{f-g}.
\end{equation}

Then combining~\eqref{AuhAu},\eqref{U1fPN},\eqref{U1fh}, we obtain estimate~\eqref{ere1} for the $\Lcl^2$-error
$ \bigl\| \Ao u_1 - \hAo u_1   \bigr\| $.

Now, consider~\eqref{ere2}. Using~\eqref{exrec} and~\eqref{hf2}, we have
\begin{equation}\label{hhf2}
  \norm{ f-\fle } = \norm{ \Gcl_d\,\Uo f - \Gcl_d\ule }\leq \norm{\Gcl_d} \cdot \norm{ \Uo f  - \ule}.
\end{equation}

Since $ \kl{ \Uo f } (x,t) = \ule(x,t) = u_1(x,t) $ for $x\in\Gm_1$, then
\begin{equation}\label{Ufhu}
  \norm{ \Uo f  - \ule} = \norm{ \Uo_2 f -\hAo u_1 } = \norm{ \Ao u_1 - \hAo u_1 }.
\end{equation}

Thus, the error estimate~\eqref{ere2} is obtained from~\eqref{hhf2}, \eqref{Ufhu}, and the error estimate~\eqref{ere1}.
\end{proof}


\begin{rem}\label{rem1}

Let $\Qcl_n\colon \Lcl^2\kl{\Om_1}\to W_n$ be the orthogonal projection on $W_n$ in the space $\Lcl^2\kl{\Om_1}$.
Then, since we have that
$ \min\limits_{g\in W_n}  \norm{f-g} = \norm{ f- \Qcl_n f }  $, we can write
$ \norm{ f- \Qcl_n f }  $ instead of $ \min\limits_{g\in W_n}  \norm{f-g} $ in~\eqref{ere1} and~\eqref{ere2}.

\end{rem}


As we see from Theorem~\ref{Th1}, the estimates of the $\Lcl^2$-errors given by our learning procedure depend on the minimal
distance from the unknown function $f$ to the linear subspace $W_n$ defined by the training functions $f_i$. This gives us
an indication for the choice of the training functions. Namely, one should choose the training functions $f_i$ such that the unknown
function $f$ can be well approximated by their linear combination.

Estimates~\eqref{ere1},\eqref{ere2} also allow us to state the condition for the exact approximation given by our learning procedure
and for the convergence of the learned approximation when the number of the training functions $n$ goes to infinity.
We present these conditions in the following two corollaries.

\begin{cor}\label{Cor1}

If $f\in W_n$, then the learned approximation $\hAo u_1$ and the reconstruction $\fle$ are exact, i.e.
$$
  \norm{ \Ao u_1 - \hAo u_1 } = \norm{ f - \fle } = 0.
$$

\end{cor}

\begin{cor}\label{Cor2}

If $\overline{ \bigcup\limits_{n\geq 1} W_n } = \Lcl^2\kl{ \Om_1 } $, then the learned approximation $\hAo u_1$ and the reconstruction
$\fle$ converge respectively to $\Ao u_1$ and $f$ as $n\to\infty$, i.e.
$$
  \lim\limits_{n\to\infty} \norm{ \Ao u_1 - \hAo u_1 } =
  \lim\limits_{n\to\infty} \norm{ f - \fle } = 0.
$$

\end{cor}



Let us now compare the errors of the approximations $\fle$ with $n\geq 1$ and $\fze$.
The $\Lcl^2$-error estimates~\eqref{ere1},\eqref{ere2} for $n=0$ become:
\begin{align}\label{ere1z}
  \norm{ \Ao u_1 -0 } & \leq
  \norm{\Ao} \cdot \norm{\Uo_1} \cdot \norm{f},
	\\
\label{ere2z}
  \norm{ f - \fze } & \leq
  \norm{\Gcl_d} \cdot \norm{\Ao} \cdot \norm{\Uo_1} \cdot \norm{f}.
\end{align}


Comparing the error estimates~\eqref{ere1},\eqref{ere2} for the learned approximations with $n\geq 1$
and the error estimates~\eqref{ere1z},\eqref{ere2z} for the approximations using zero extension of the limited view
wave data, one sees that these error estimates differ
regarding
the following factors:
\begin{equation}\label{Eleze}
  \Ele(f) := \min\limits_{g\in W_n} \norm{ f-g },\quad
  \Eze(f):= \norm{f},
\end{equation}
correspondingly for learned approximations with $n\geq 1$ and approximation using zero extension.


The factors~\eqref{Eleze} can be seen as indicators for the expected approximation quality of the considered algorithms.
For a fixed non-zero function $f$, the factor $\Eze(f)$ is a fixed non-zero value, while the factor $\Ele(f)$ can be zero, or can be made
arbitrary small, see Corollaries~\ref{Cor1},\ref{Cor2}. Therefore, the approximation quality of the learned
approximations
is expected to be
better than of the approximations using zero extension of the data. This expectation will be confirmed by the
numerical results in the next section.
In fact, one can show (see Remark~\ref{Rem2} below) that
the factor $\Ele(f)$ is always less or equal than the factor $\Eze(f)$, and
the strict inequality $\Ele(f) < \Eze(f)$ holds under rather mild conditions on the function $f$ and the training functions $f_i$.
Generally, this condition can be expected to hold in practice.


\begin{rem}\label{Rem2}

Using properties of the projection operators in Hilbert spaces, one can show that
the sequence $\Ele(f)$ is nonincreasing, i.e.
\begin{equation}\label{rel1rm}
  \Ecl_n(f) \leq \Ecl_m(f) \quad\mbox{for}\quad n>m\geq 0.
\end{equation}
If additionally
\begin{equation}\label{fficrm}
  \inner{ f,f_i } \neq 0\quad\mbox{for some}\quad i\in\Set{ m+1,\ldots,n },
\end{equation}
then inequality~\eqref{rel1rm} is strict, i.e.
\begin{equation}\label{rel2rm}
  \Ecl_n(f) < \Ecl_m(f) \quad\mbox{for}\quad n>m\geq 0.
\end{equation}
Condition~\eqref{fficrm} is also necessary for~\eqref{rel2rm}, i.e. if~\eqref{rel2rm} holds, then we have~\eqref{fficrm}.

\end{rem}


\section{Numerical results}\label{s:NR}


In this section, we present results of the numerical realization of the proposed operator learning approach.

We consider the spatial dimension $d=2$, and we take the elliptical domain
$$
  \Om = \Set{ \kl{x_1,x_2}\in\R^2  |  \kl{ x_1/a_1 }^2 +  \kl{ x_2/a_2 }^2  < 1  },
$$
with $a_1=2$, $a_2=1$. We use the following parametrization of the boundary
$$
  \pOm = \Set{ \kl{ a_1 \cos\theta, a_2\sin\theta }   |   \theta\in[ -\pi,\pi )  },
$$
and we assume that the unobservable part of the boundary is (see Figure~\ref{Fig:pha}(left))
$$
  \Gm_2 = \Set{ \kl{ a_1 \cos\theta, a_2\sin\theta }   |   \theta\in[ 0.97, 2.17 )  }.
$$
\rtb
Thus, approximately 19\% of the angular   values are missing.
\rte

We work with the function $f$ presented in Figure~\ref{Fig:pha}(left). Its numerical full view wave boundary data $u=\Ucl f$ is given
in Figure~\ref{Fig:pha}(right), and we use the corresponding limited view wave boundary data $u_1 = \Ucl_1 f$.
The observation boundary $\Gm_1$ is discretized such that the distance between two consecutive points is in the interval
[0.0099, 0.0101]. We take the time step size as 0.01.


\begin{figure}[t]
\centering

\begin{tabular}{cc}
\parbox[c]{7.0cm}{
\includegraphics[width=7cm]{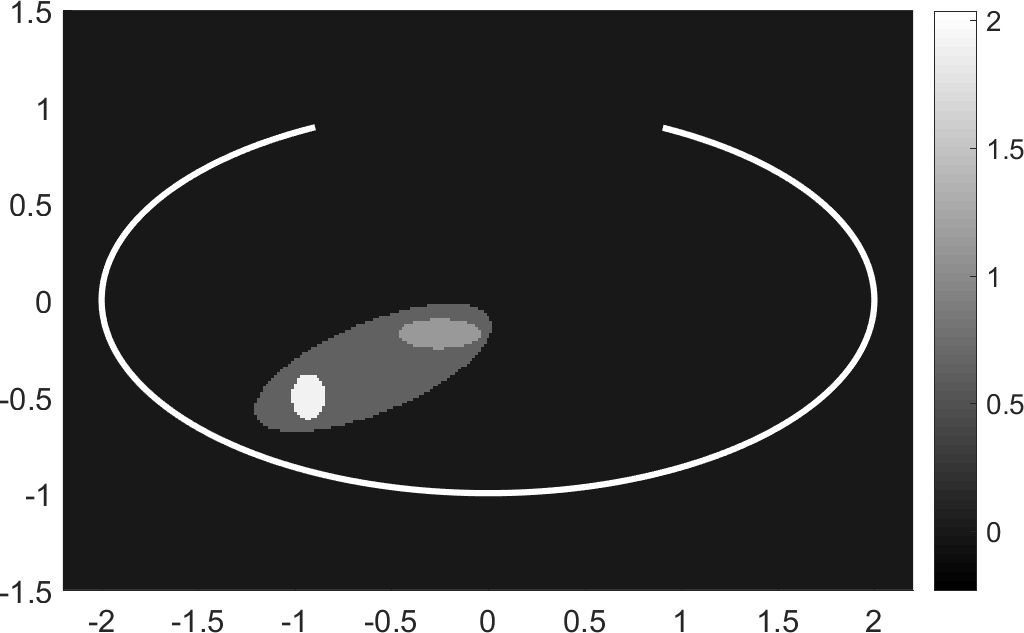}
}
&
\parbox[c]{7.5cm}{
\includegraphics[width=7.5cm]{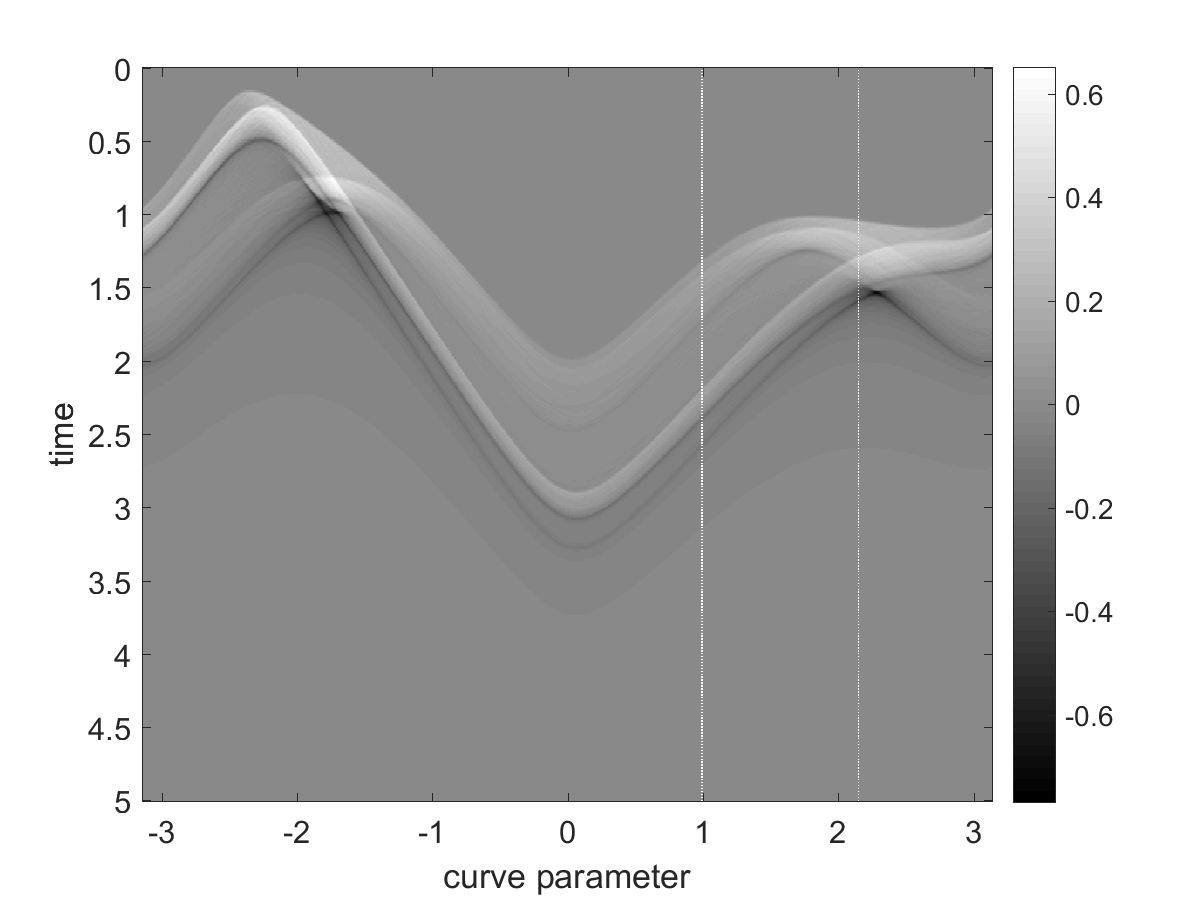}
}
\end{tabular}

\caption{
Left: the function $f$ that we use in our numerical experiments and the chosen observation boundary $\Gm_1$.
Right: the corresponding numerical full view wave boundary data $\Uo f$. The region between two white vertical
lines corresponds to the unknown part of the data on the unobservable part of the boundary $\Gm_2$.
\label{Fig:pha}
}

\end{figure}


We further assume that we know a rectangular region
$$
  K=\Set{ \kl{x_1,x_2}\in\R^2 |  -1.25 \leq x_1 < 0.5,\; -0.7 \leq x_2 < 0.1752 }
$$
that contains the support of $f$
(Figure~\ref{Fig:set}(top and bottom)).
We use this region $K$ for defining training functions $f_i$.
Namely, we consider partitions of the region $K$ into squares $K_i$, $i\in\Set{ 1,\ldots,n }$.
The square $K_i$ contains points $\kl{ x_1,x_2 }\in\R^2$ such that
\begin{align*}
-1.25+ \kl{   \left\lceil i/n_h  \right\rceil -1 }   w/n_w
\leq &\; x_1 < -1.25+      \left\lceil     i/n_h   \right\rceil   w/n_w    ,
\\
-0.7+     \kl{       i \ \mathrm{ mod } \ n_h -1    }    h/n_h
\leq &\; x_2 < -0.7+       \kl{        i \ \mathrm{mod} \ n_h         }       h/n_h,
\end{align*}
where $w=1.75$ (width of $K$), $h=0.8752$ (height of $K$), $n_w = \sqrt{ 2n }$, $n_h = n_w / 2$
(see Figure~\ref{Fig:set}(middle)).
Then we define the training function $f_i$ as the indicator function of the square $K_i$.
We take the number of the training functions in the form
$n=n_1\times n_2$, where $n_1$ and $n_2$ are the numbers of the partitioning intervals along the coordinate $x_1$ and $x_2$
correspondingly. We present the numerical results for $n=4\times 2,\; 8\times 4,\; 16\times 8,\; 32\times 16$.

\rtb
Let us note that we use the rectangular region $K$ for illustration purpose.
If the region containing $\supp(f)$ is not known, then one may consider squares
filling
the whole subset
$\Om_1$ of the detection region $\Dbb\kl{ \Gm_1 }$.
Further note that other type of basis functions can be used in a similar manner.
Kaiser-Bessel functions, which are frequently used in computed tomography
(see, e.g., \cite{MatLew96,WanSchSOA14,SchPerHal17}),
would be another reasonable choice.
\rte


\begin{figure}[t]
\centering
\begin{tikzpicture}[scale=1.5]
			\draw[domain=pi-rad(atan(2)):2*pi+rad(atan(2)),samples=300,variable=\x] plot({2*cos(\x r)},{sin(\x r)});
			\draw [fill=lightgray,draw=lightgray] (-1.25,-0.7) rectangle (0.5,0.1752);
			\draw[scale=1] (-0.375,-0.2624) node{$K$};
			\draw[dashed,domain=rad(atan(2)):pi-rad(atan(2)),samples=300,variable=\x] plot({2*cos(\x r)},{sin(\x r)});
			\draw[dotted,very thin] (${2*cos(atan(2))}*(1,0)+{sin(atan(2))}*(0,1)$)--(${(-2)*cos(atan(2)}*(1,0)+{sin(atan(2))}*(0,1)$);
			\draw (-1.25,0.5) node[]{$\Omega$};
			\draw (0,1.2) node[]{$\Gamma_2$};
			\draw (1.75,-0.8) node[]{$\Gamma_1$};
\end{tikzpicture}

\vspace*{0.8\baselineskip}

\hspace*{2.1cm}
\begin{tikzpicture}[scale=1.5]
			\draw[domain=pi-rad(atan(2)):2*pi+rad(atan(2)),samples=300,variable=\x] plot({2*cos(\x r)},{sin(\x r)});
			\foreach \x in {0,...,7}
				\foreach \y in {0,...,3}
					\draw [gray,very thin] (${-1.25+\x*(1.75/8)}*(1,0)+{-0.7+\y*(0.8752/4)}*(0,1)$) rectangle (${-1.25+(\x+1)*(1.75/8)}*(1,0)+{-0.7+(\y+1)*(0.8752/4)}*(0,1)$);
			\draw [fill=lightgreen,draw=lightgreen] (${-1.25}*(1,0)+{-0.7}*(0,1)$) rectangle (${-1.25+(1.75/8)}*(1,0)+{-0.7+(0.8752/4)}*(0,1)$);
			\draw [green,thick] (${-1.25}*(1,0)+{-0.7}*(0,1)$)--(${-1.25+0.21875}*(1,0)+{-0.7}*(0,1)$);
			\draw [green,thick] (${-1.25}*(1,0)+{-0.7}*(0,1)$) -- (${-1.25}*(1,0)+{-0.7+(0.8752/4)}*(0,1)$);
			\draw [fill=lightgray,draw=lightgray] (${-1.25}*(1,0)+{-0.7+(0.8752/4)}*(0,1)$) rectangle (${-1.25+(1.75/8)}*(1,0)+{-0.7+2*(0.8752/4)}*(0,1)$);
			\draw [gray,thick] (${-1.25}*(1,0)+{-0.7+(0.8752/4)}*(0,1)$)--(${-1.25+0.21875}*(1,0)+{-0.7+(0.8752/4)}*(0,1)$);
			\draw [gray,thick] (${-1.25}*(1,0)+{-0.7+(0.8752/4)}*(0,1)$) -- (${-1.25}*(1,0)+{-0.7+2*(0.8752/4)}*(0,1)$);
			\draw [fill=lightorange,draw=lightorange] (${-1.25}*(1,0)+{-0.7+2*(0.8752/4)}*(0,1)$) rectangle (${-1.25+(1.75/8)}*(1,0)+{-0.7+3*(0.8752/4)}*(0,1)$);
			\draw [darkorange,thick] (${-1.25}*(1,0)+{-0.7+2*(0.8752/4)}*(0,1)$)--(${-1.25+0.21875}*(1,0)+{-0.7+2*(0.8752/4)}*(0,1)$);
			\draw [darkorange,thick] (${-1.25}*(1,0)+{-0.7+2*(0.8752/4)}*(0,1)$) -- (${-1.25}*(1,0)+{-0.7+3*(0.8752/4)}*(0,1)$);
			\draw[dashed,domain=rad(atan(2)):pi-rad(atan(2)),samples=300,variable=\x] plot({2*cos(\x r)},{sin(\x r)});
			\draw[dotted,very thin] (${2*cos(atan(2))}*(1,0)+{sin(atan(2))}*(0,1)$)--(${(-2)*cos(atan(2)}*(1,0)+{sin(atan(2))}*(0,1)$);
			
			\draw [fill=lightorange,draw=lightorange] (${-1.25+3.5}*(1,0)+{-0.7+0.8}*(0,1)$) rectangle (${-1.25+(1.75/8)+3.5}*(1,0)+{-0.7+(0.8752/4)+0.8}*(0,1)$);
			\draw (${-1.25+(1.75/8)+3.5}*(1,0)+{-0.7+0.5*(0.8752/4)+0.8}*(0,1)$) node[anchor=west]{$\mathrm{supp}(f_3)$};
			\draw [darkorange,thick] (${-1.25+3.5}*(1,0)+{-0.7+0.8}*(0,1)$)--(${-1.25+0.21875+3.5}*(1,0)+{-0.7+0.8}*(0,1)$);
			\draw [darkorange,thick] (${-1.25+3.5}*(1,0)+{-0.7+0.8}*(0,1)$) -- (${-1.25+3.5}*(1,0)+{-0.7+(0.8752/4)+0.8}*(0,1)$);
			\draw [fill=lightgray,draw=lightgray] (${-1.25+3.5}*(1,0)+{-0.7+(0.8752/4)+0.9}*(0,1)$) rectangle (${-1.25+(1.75/8)+3.5}*(1,0)+{-0.7+2*(0.8752/4)+0.9}*(0,1)$);
			\draw (${-1.25+(1.75/8)+3.5}*(1,0)+{-0.7+1.5*(0.8752/4)+0.9}*(0,1)$) node[anchor=west]{$\mathrm{supp}(f_2)$};
			\draw [gray,thick] (${-1.25+3.5}*(1,0)+{-0.7+(0.8752/4)+0.9}*(0,1)$)--(${-1.25+0.21875+3.5}*(1,0)+{-0.7+(0.8752/4)+0.9}*(0,1)$);
			\draw [gray,thick] (${-1.25+3.5}*(1,0)+{-0.7+(0.8752/4)+0.9}*(0,1)$) -- (${-1.25+3.5}*(1,0)+{-0.7+2*(0.8752/4)+0.9}*(0,1)$);
			\draw [fill=lightgreen,draw=lightgreen] (${-1.25+3.5}*(1,0)+{-0.7+2*(0.8752/4)+1}*(0,1)$) rectangle (${-1.25+(1.75/8)+3.5}*(1,0)+{-0.7+3*(0.8752/4)+1}*(0,1)$);
			\draw (${-1.25+(1.75/8)+3.5}*(1,0)+{-0.7+2.5*(0.8752/4)+1}*(0,1)$) node[anchor=west]{$\mathrm{supp}(f_1)$};
			\draw [green,thick] (${-1.25+3.5}*(1,0)+{-0.7+2*(0.8752/4)+1}*(0,1)$)--(${-1.25+0.21875+3.5}*(1,0)+{-0.7+2*(0.8752/4)+1}*(0,1)$);
			\draw [green,thick] (${-1.25+3.5}*(1,0)+{-0.7+2*(0.8752/4)+1}*(0,1)$) -- (${-1.25+3.5}*(1,0)+{-0.7+3*(0.8752/4)+1}*(0,1)$);
			\draw (-1.25,0.5) node[]{$\Omega$};
			\draw (0,1.2) node[]{$\Gamma_2$};
			\draw (1.75,-0.8) node[]{$\Gamma_1$};
\end{tikzpicture}

\vspace*{0.8\baselineskip}

\begin{tikzpicture}[scale=1.5]
			\draw[domain=pi-rad(atan(2)):2*pi+rad(atan(2)),samples=300,variable=\x] plot({2*cos(\x r)},{sin(\x r)});
			\foreach \x in {0,...,7}
				\foreach \y in {0,...,3}
					\draw [gray,very thin] (${-1.25+\x*(1.75/8)}*(1,0)+{-0.7+\y*(0.8752/4)}*(0,1)$) rectangle (${-1.25+(\x+1)*(1.75/8)}*(1,0)+{-0.7+(\y+1)*(0.8752/4)}*(0,1)$);
			\draw[fill=lightgray,draw=lightgray,domain=0:2*pi,samples=300,variable=\x] plot({(0.65625*cos(\x r))*cos(pi/8 r)-(0.8752/3*sin(\x r))*sin(pi/8 r)-0.59375},{(0.65625*cos(\x r))*sin(pi/8 r)+(0.8752/3*sin(\x r))*cos(pi/8 r)-0.2624});
			\draw[scale=1] (-0.55,-0.0125) node[anchor=north]{$\mathrm{supp}(f)$};
			\draw[dashed,domain=rad(atan(2)):pi-rad(atan(2)),samples=300,variable=\x] plot({2*cos(\x r)},{sin(\x r)});
			\draw[dotted,very thin] (${2*cos(atan(2))}*(1,0)+{sin(atan(2))}*(0,1)$)--(${(-2)*cos(atan(2)}*(1,0)+{sin(atan(2))}*(0,1)$);
			\draw (-1.25,0.5) node[]{$\Omega$};
			\draw (0,1.2) node[]{$\Gamma_2$};
			\draw (1.75,-0.8) node[]{$\Gamma_1$};
\end{tikzpicture}
\caption{
Top: the rectangular region $K$ containing $\supp(f)$. Middle: the example of the partition of $K$ into $8\times 4$ squares.
The training functions $f_i$ are numbered starting from the bottom-left square from bottom to top and from left to right.
Bottom: the position of $\supp(f)$ in $K$ with the partition of $K$ into $8\times 4$ squares.
\label{Fig:set}
}
\end{figure}
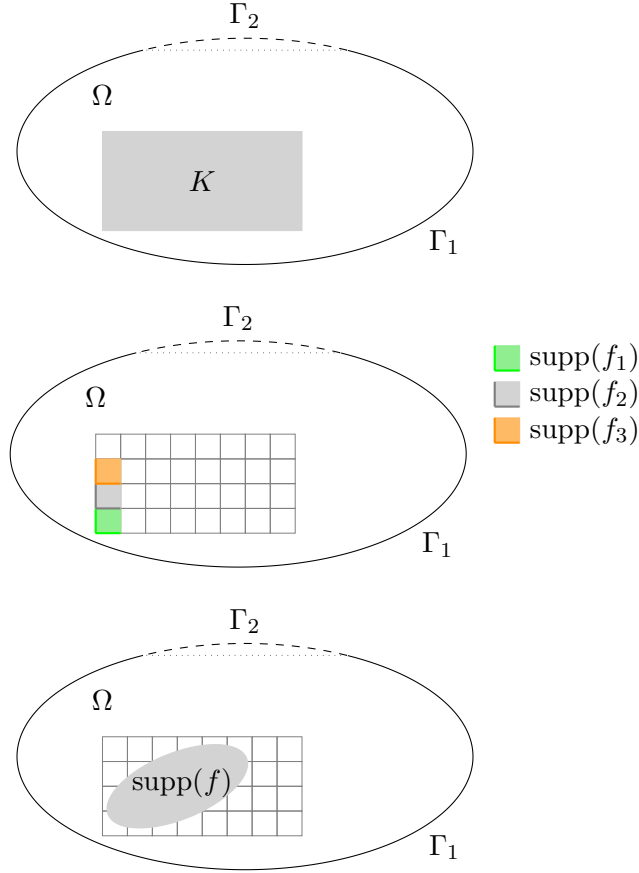	
	

The extended limited view data $\ule$ using the learned extension operator $\hAo$ for the considered values of $n$
are presented in Figure~\ref{Fig:estpr}. We observe that as $n$ increases, the extended data $\ule$ approaches the full
view data $u$ in Figure~\ref{Fig:pha}(right). Note that the chosen training functions $f_i$ satisfy the condition of Corollary~\ref{Cor2}.
Therefore, the approach of $\ule$ to the full view data $u$ is in agreement with our theoretical analysis.


\begin{figure}[t]
\centering

\begin{tabular}{cc}
\includegraphics[width=6cm]{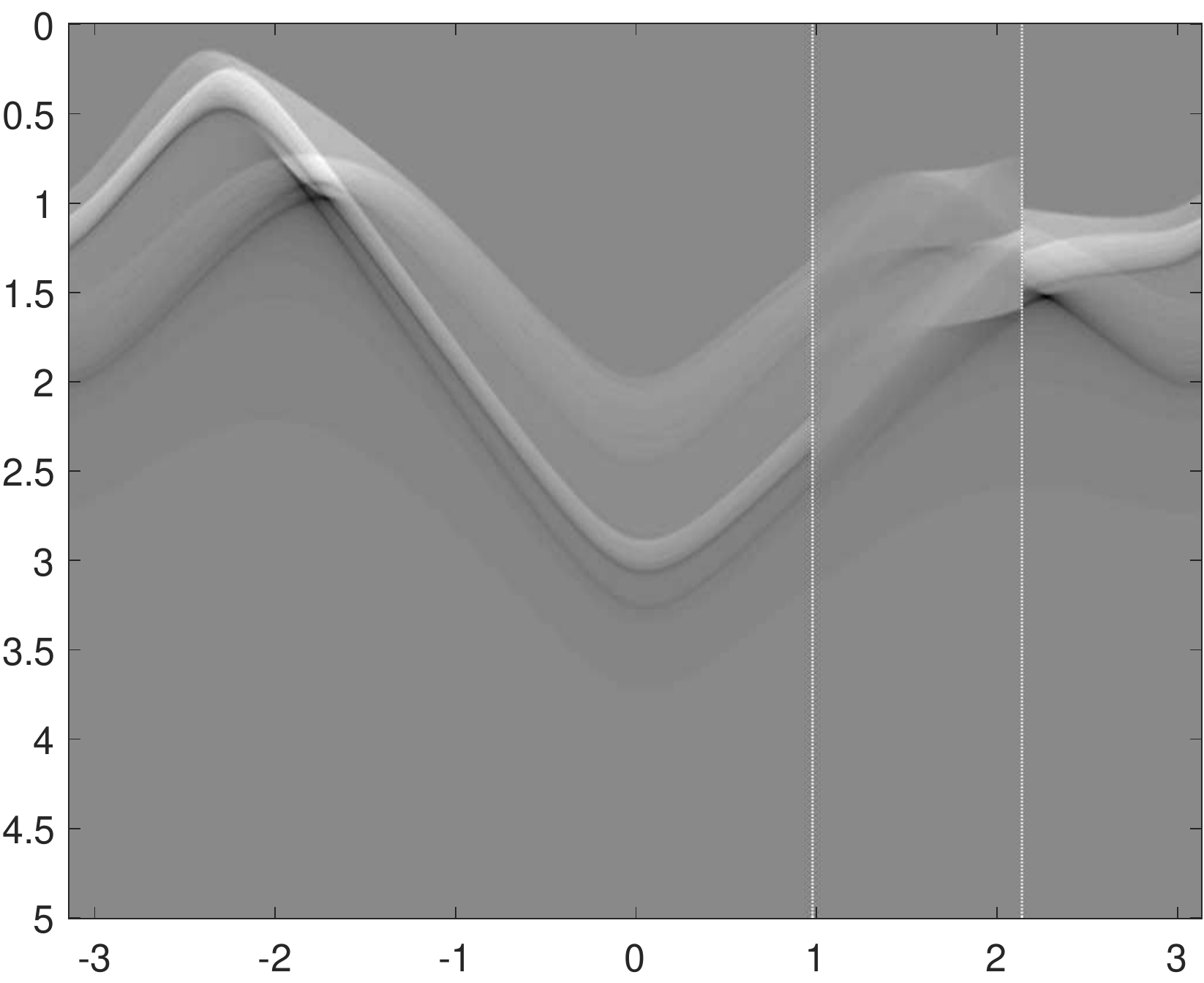}
&
\includegraphics[width=6cm]{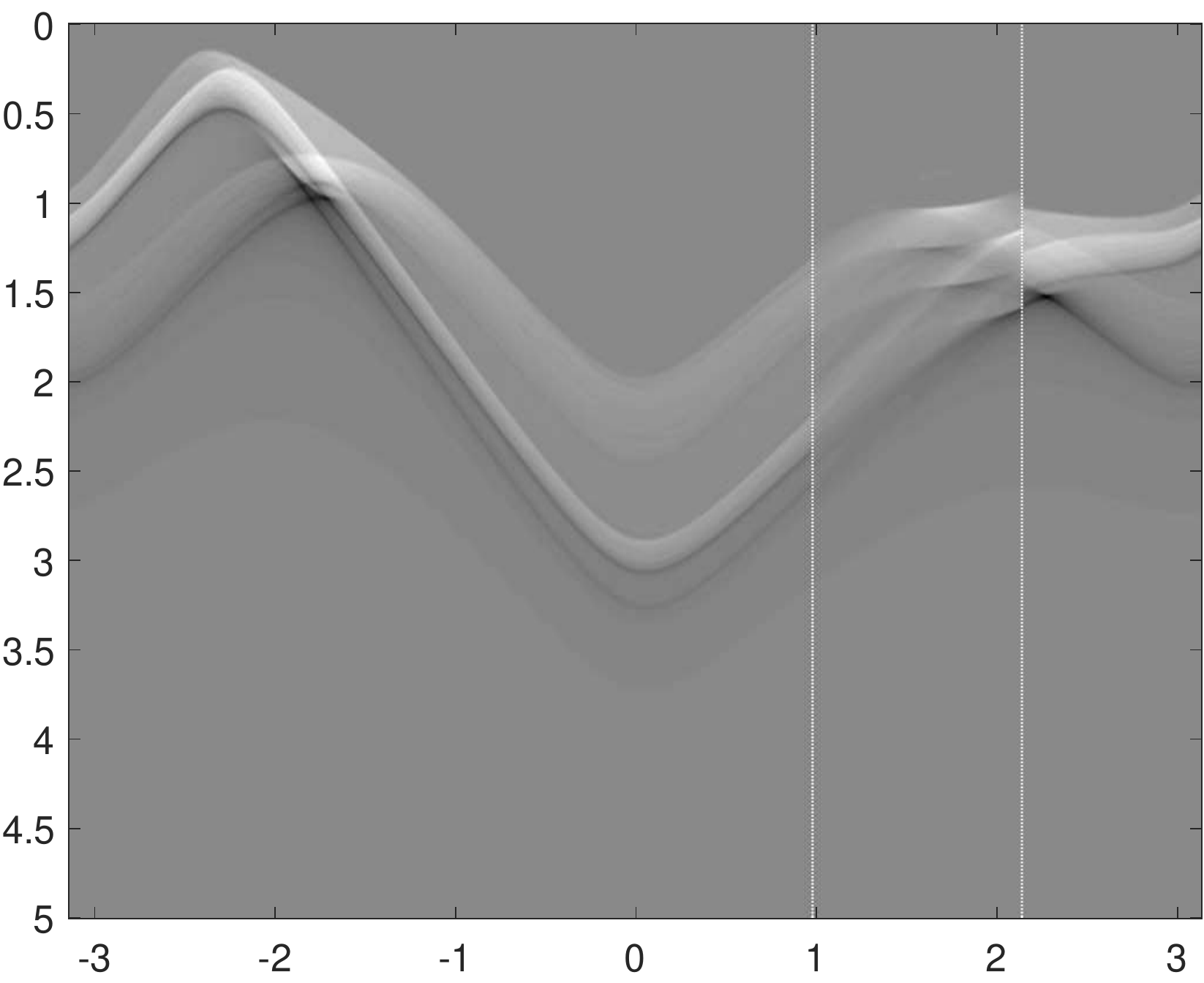} \\
\includegraphics[width=6cm]{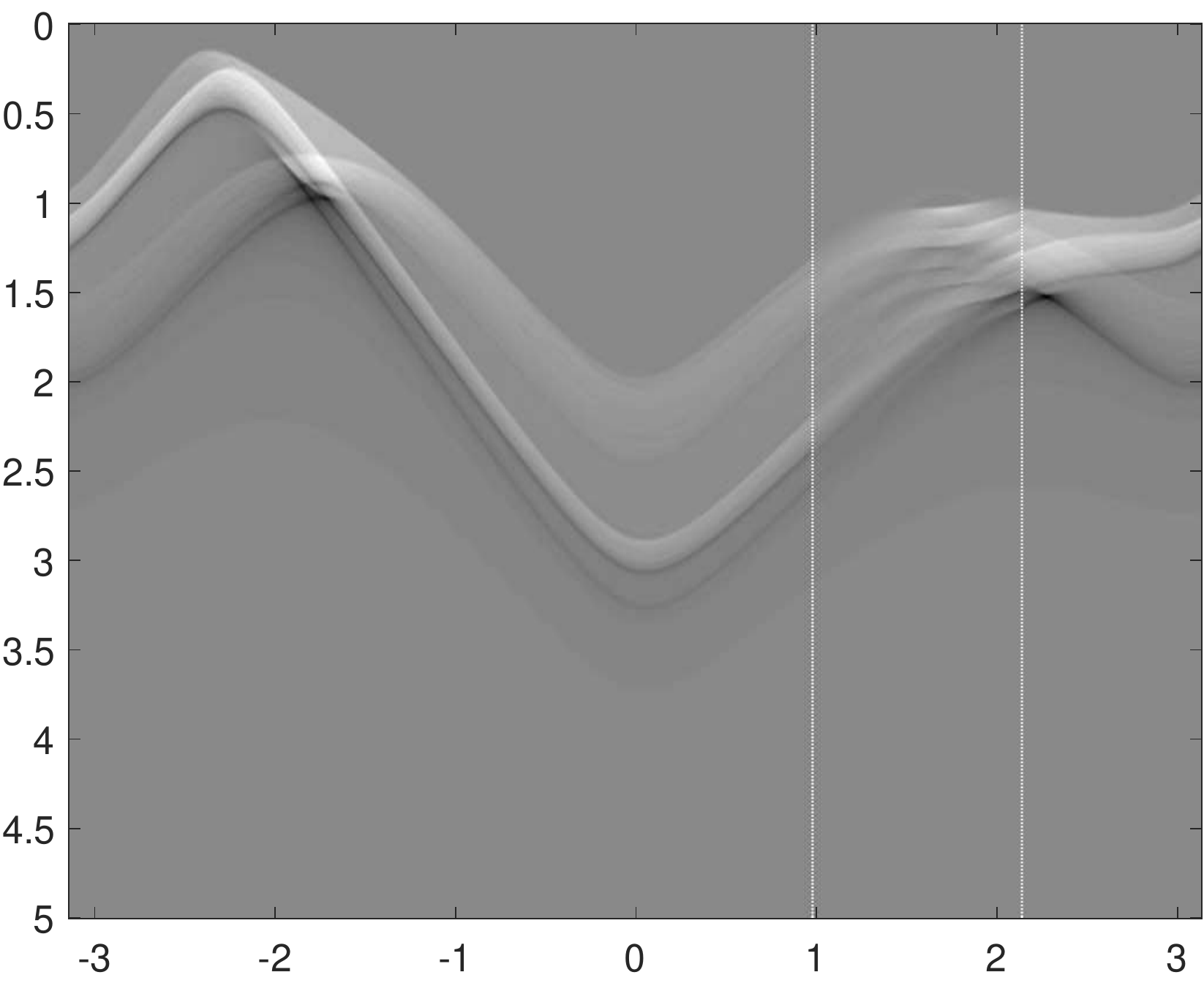}
&
\includegraphics[width=6cm]{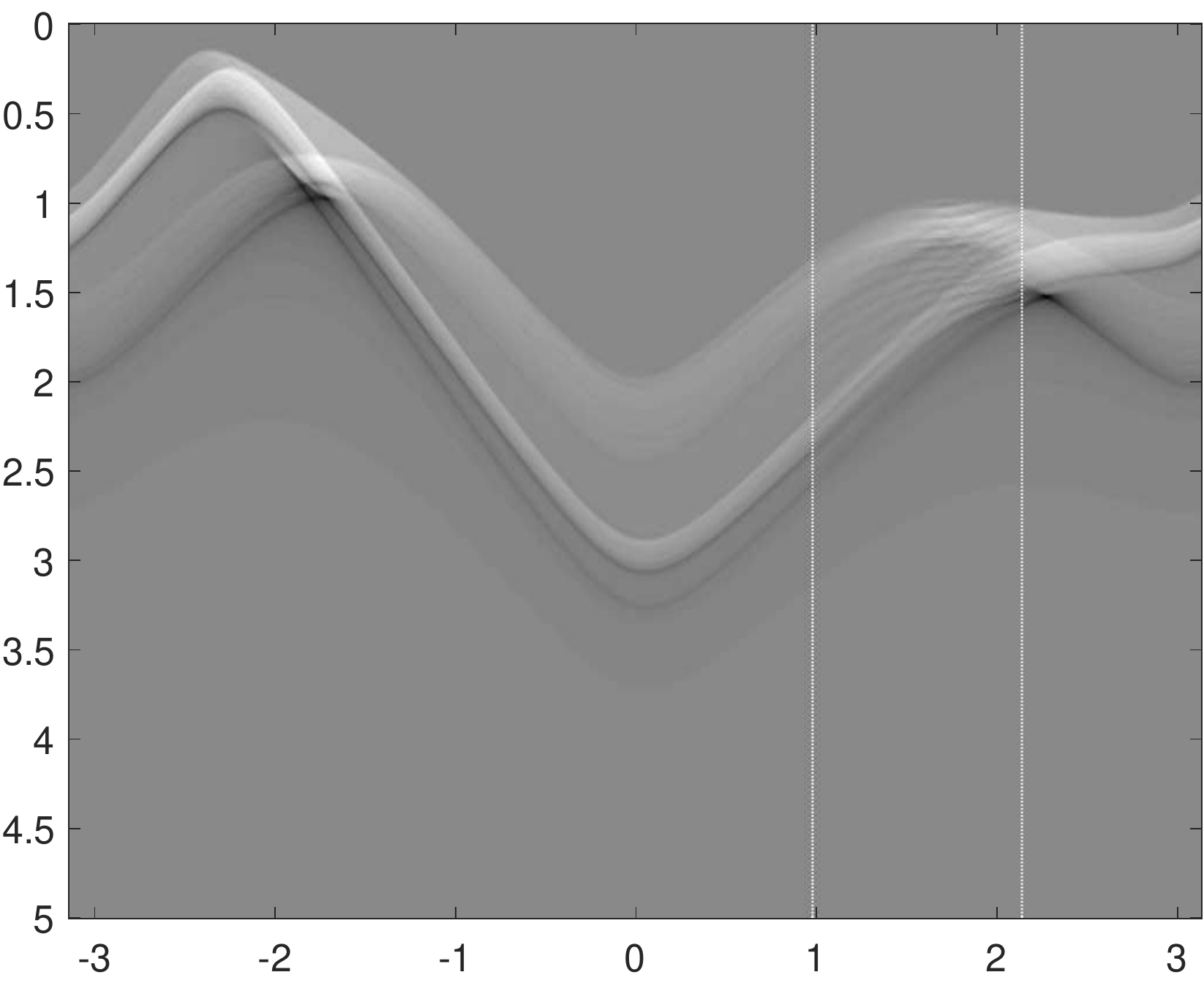}
\end{tabular}

\caption{
The extended limited view data $\ule$ using the learned extension operator $\hAo$ for
$n=4\times 2,\; 8\times 4,\; 16\times 8,\; 32\times 16$ (from left to right and from top to
bottom). The gray scaling is as in Figure~\ref{Fig:pha}(right).
\label{Fig:estpr}}

\end{figure}


The reconstructions $\fle$ using the extended data $\ule$ are presented in Figure~\ref{Fig:recs}(2nd and 3rd rows).
For comparison purpose, we also present the reconstruction $\hat f$ using the full view wave boundary data $u$, and the
reconstruction $\fze$ using the zero extended data $\uze$ (Figure~\ref{Fig:recs}(1st row)).
We evaluate the reconstructions at the points from the discrete set
$$
  \Om_h := \Set{ \kl{ -2.2 + n_1 h, -2.2 + n_2 h }\in\R^2 |  n_1,n_2\in\Set{ 0,1,\ldots,300 }    } \cap \Om,
$$
with $h=11/750$. We also consider the discrete $\Lcl^2$-error of a reconstruction $\hat f_*$ defined as follows:
$$
  E_2\kl{ \hat f_* } := \kl{  \sum\limits_{x\in\Om_h}
            \abs{ f(x) - \hat f_*(x) }^2\cdot h^2        }^{1/2} .
$$

Let us discuss the reconstructions in Figure~\ref{Fig:recs}. First of all, as expected, one observes strong artifacts
in the reconstruction $\fze$, especially outside of $\supp(f)$. These artifacts are considerably corrected in the reconstruction
$\hat f_{ 4\times 2 }$, and as the number of the training functions $n$ increases, the artifacts become weaker such that
the reconstruction $\hat f_{ 32\times 16 }$ is very similar to the reconstruction $\hat f$. This observation is also reflected
in $E_2$-errors that are presented in Figure~\ref{Fig:err}. Note that $\hat f$ differs from $f$ due to the discretization error of the
numerical realization of the formula $\Gcl_2$. Thus, as in the case of the data $\ule$, the approach of $\fle$ to $f$ is in 
agreement with Corollary~\ref{Cor2}.


\begin{figure}[t]
\centering
\begin{tabular}{cc}
\includegraphics[width=6cm]{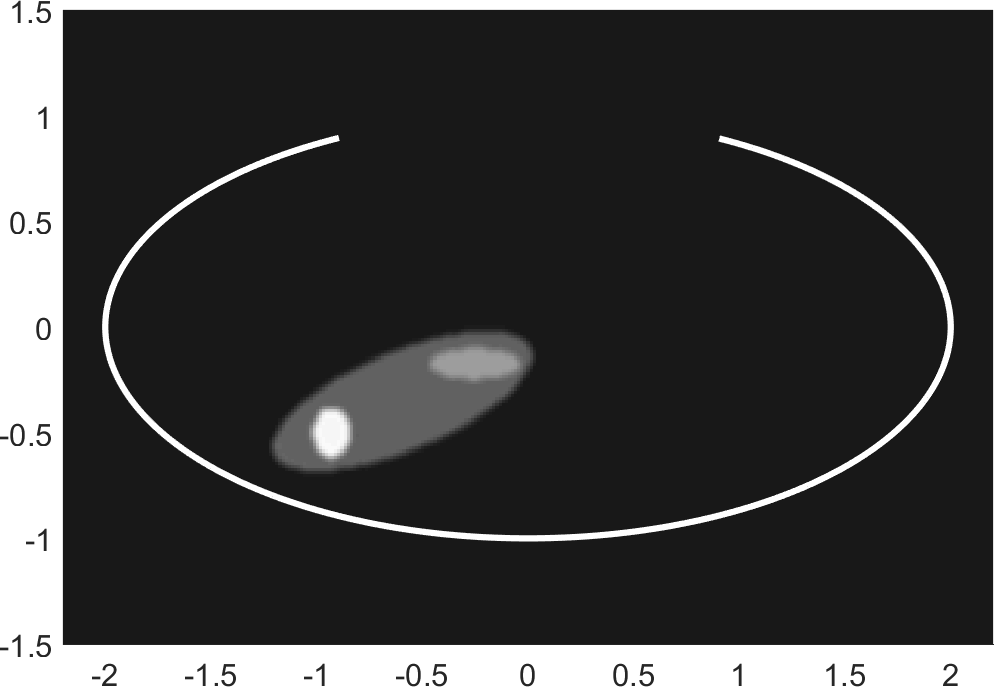}
&
\includegraphics[width=6cm]{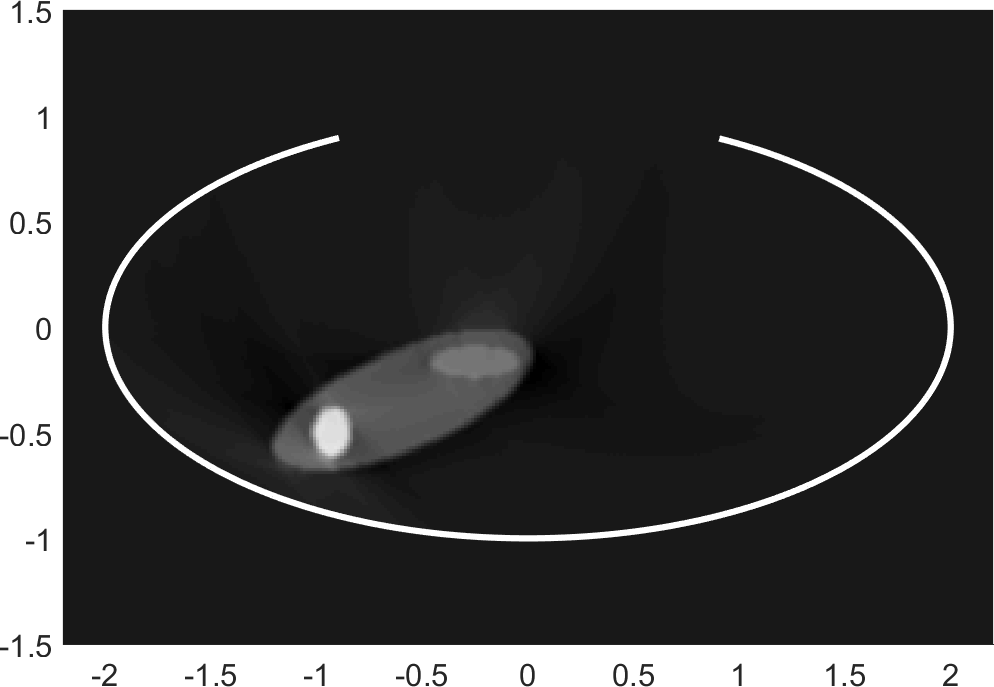} \\
\includegraphics[width=6cm]{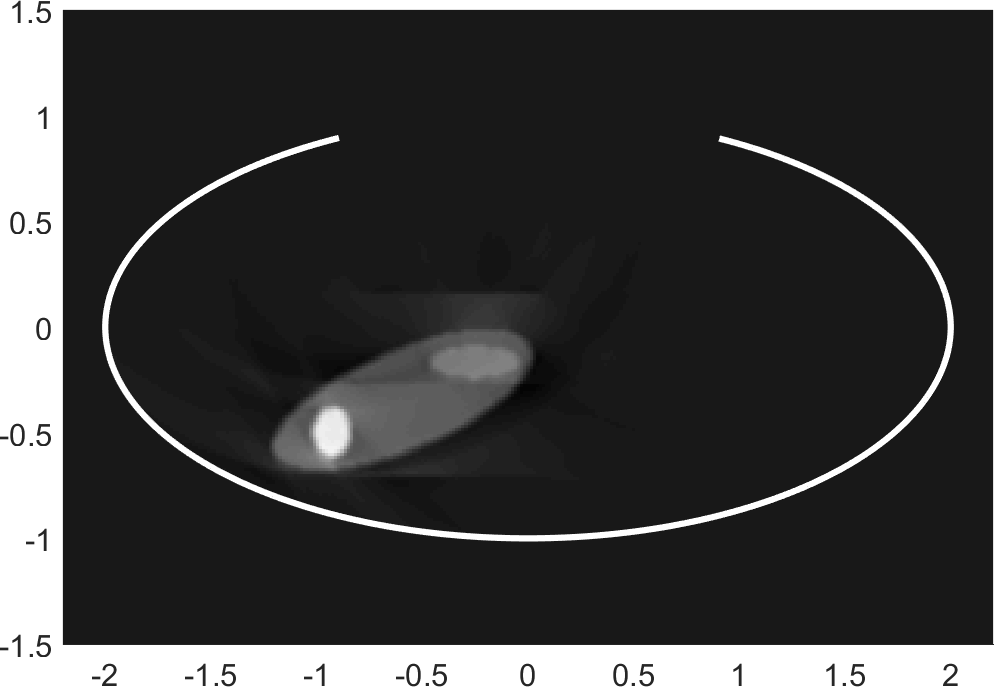}
&
\includegraphics[width=6cm]{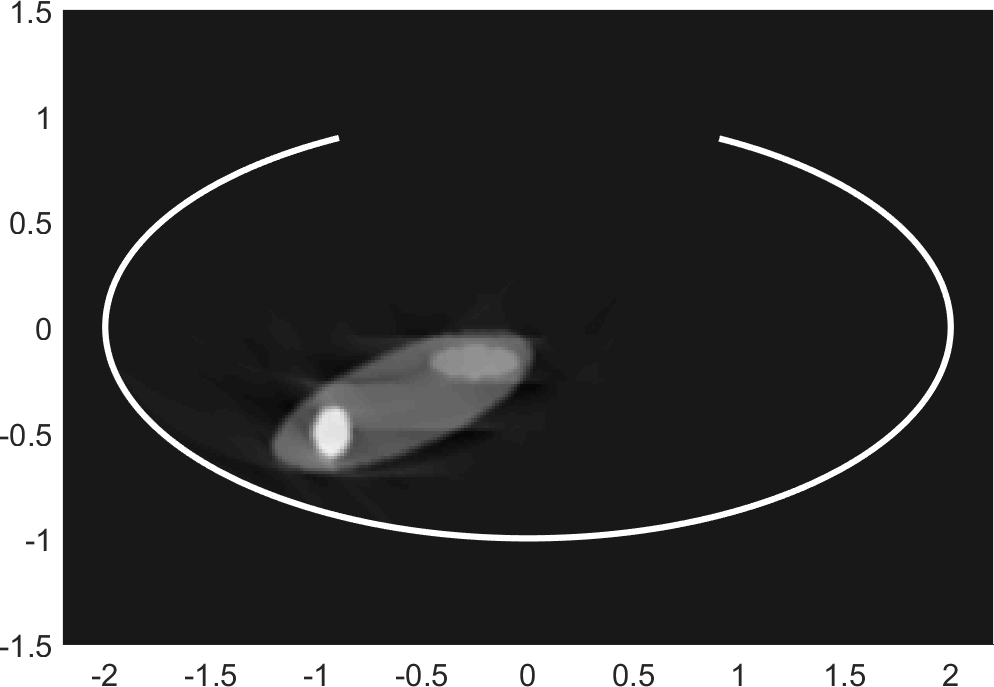} \\
\includegraphics[width=6cm]{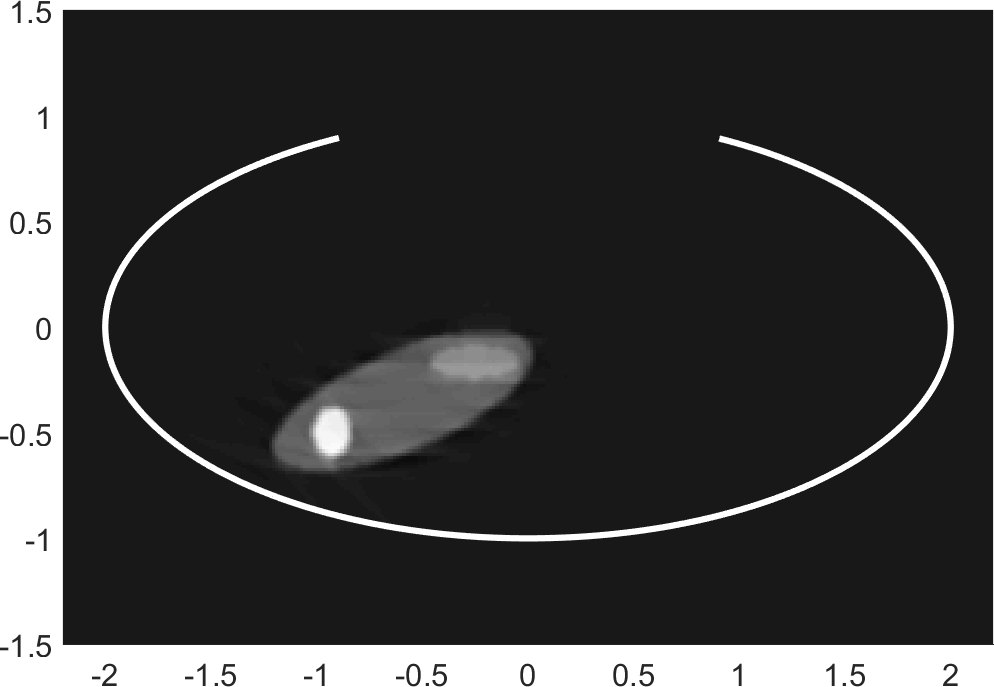}
&
\includegraphics[width=6cm]{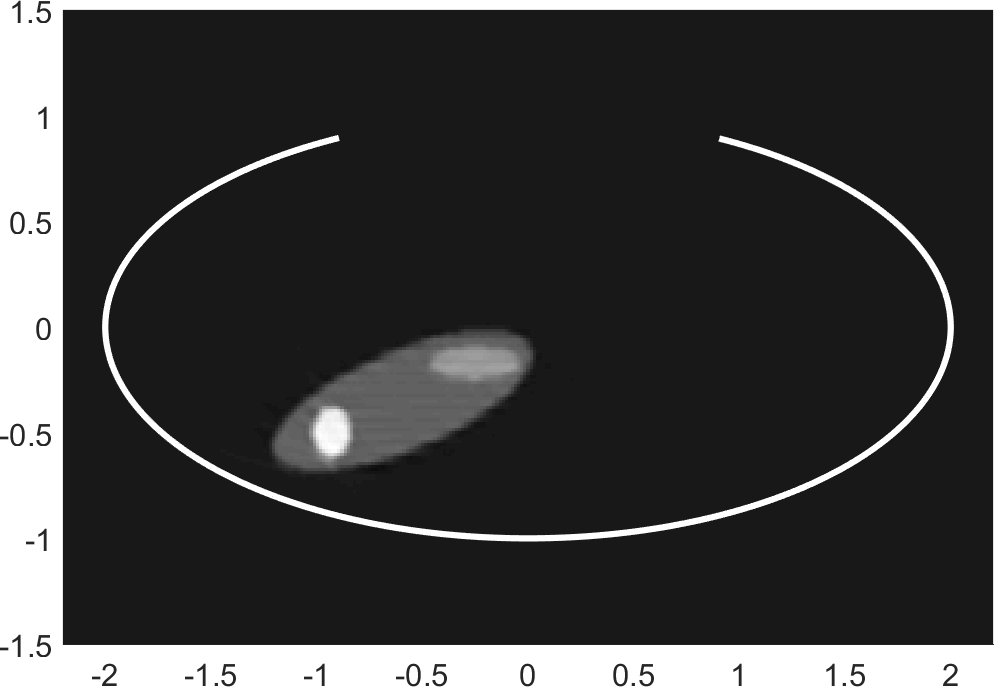}
\end{tabular}
\caption{
From left to right and from top to bottom: the reconstructions
$\hat f$, $\fze$, and $\fle$, for
$n=4\times 2,\; 8\times 4,\; 16\times 8,\; 32\times 16$.
The gray scaling is as in Figure~\ref{Fig:pha}(left).
\label{Fig:recs}}
\end{figure}


\begin{figure}[t]
		\centering
		\includegraphics[scale=0.5]{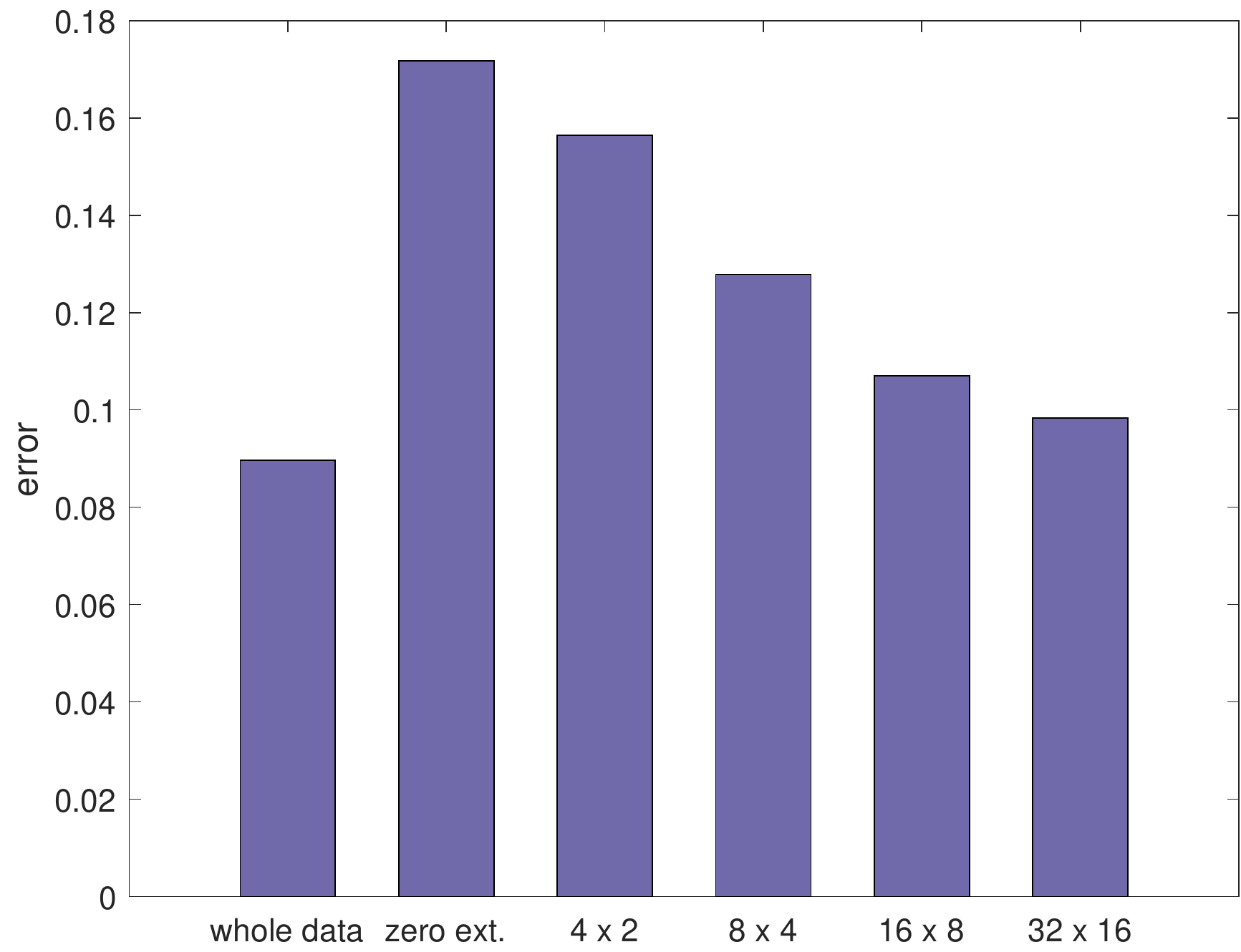}
\caption{
$E_2$-errors of the considered reconstructions
$\hat f$, $\fze$, and $\fle$, for
$n=4\times 2,\; 8\times 4,\; 16\times 8,\; 32\times 16$.
\label{Fig:err}}
\end{figure}


Finally, in Table~\ref{Tab:ctms}, we present the calculation times for the parts involved in the proposed reconstruction approach.
Our numerical results are performed with MATLAB version R2015b on the PC lenovo e31 with four processors Intel(R) Xeon(R) CPU 3.20GHz.
We see that the most time consuming part is the calculation of the matrix $\Pm_n^{-1}$, which is used for solving the system of linear
equations~\eqref{sysline}.
\rtb
Here, the calculation of $u_{1,i} = \Ucl_1 f_i$ is the most computationally expensive.
\rte
But for a given set of the training functions $f_i$,
$u_{1,i}$ and the matrix  $\Pm_n^{-1}$   have to be calculated only once
and prior to the actual image reconstruction process.

The calculation of the learned data extension $\hAo u_1$ is fast. In particular, for the biggest considered number $n=32\times 16$
of the training functions, the calculation time for $\hAo u_1$ is near the calculation time for the formula $\Gcl_2$. Thus, our proposed
operator learning approach fulfills the requirements that we stated at the beginning of Section~\ref{s:LEO}. Namely, the closeness
of the approximation $\hAo u_1$ to $\Ao u_1$, and the fast evaluation of $\hAo u_1$ are realized.

\begin{table}[t]
		\caption{
		Calculation times in seconds for the parts involved in the proposed reconstruction approach.
		}
		\label{Tab:ctms}
		\centering
		\begin{tabular}{|c|c|c|c|}
		\hline
			$n$ & $\Pm_n^{-1}$ &
			\parbox[c]{1.2cm}{\vspace*{0.1cm} \mc $\hAo u_1$\vspace*{0.2cm}}
			& $\Gcl_2$ \\ \hline
			$4 \times 2$ & $1179.73$ & $0.53$ & $4.20$\\
			$8 \times 4$ & $4707.31$ & $0.68$ & $3.55$ \\
			$16 \times 8$ & $19036.23$ & $1.41$ & $3.90$\\
			$32 \times 16$ & $75874.87$ & $6.07$ & $4.33$ \\ \hline
		\end{tabular}
\end{table}

\section{Conclusion and outlook}\label{s:Con}

In this paper, we demonstrated that an approximate extension of the limited view data
in PAT can be realized using an operator learning approach.
Our numerical results show that the learned extension of the limited view data with a
\rtba good \rte approximation quality
and a low computational cost is possible.
\rtba
A good approximation quality is especially achieved for the biggest number $n=32\times 16$
of considered training functions.
\rte
This makes the proposed learned data extension attractive for the algorithms that are
designed for the full view data. As an example, we demonstrated a satisfactory performance of a reconstruction formula
with the proposed learned data extension.

It could be interesting to look at the behavior of the proposed learned data extension without knowledge of a rectangular
region $K$ containing $\supp(f)$.
As we already noted,
in this case, one could consider partitions of the whole detection region $\Om_1$.
Also other training functions, such as
generalized Kaiser-Bessel functions (see, e.g., \cite{MatLew96,WanSchSOA14,SchPerHal17}), can be tried.

It is appealing to consider a comparison of the reconstruction quality and computation time of the proposed reconstruction
approach and iterative reconstruction algorithms.
Implementation of the proposed learned extension of the limited view
data to three spatial dimensions is an interesting aspect of future research.
\rtb
In this case, the choice of the generalized Kaiser-Bessel functions as the training functions $f_i$ is particularly convenient
because for them the wave data
$u_{1,i} = \Ucl_1 f_i$, $u_{2,i} = \Ucl_2 f_i$
are known analytically
(see, e.g., \cite{WanSchSOA14}). This makes the determination of the entries of the matrix
$\Pm_n$ fast. Also,  the solution of the system of linear equations~\eqref{sysline} can be done either using
iterative methods, such as conjugate gradient method, or an approximate inverse matrix to $\Pm_n$ can be determined.
\rte

Finally, it seems to be worth to examine applications of the presented operator learning approach to the
limited data problems in other tomographic modalities, such as sparse angle or region of interest computed tomography.

\section*{Acknowledgements}	

Authors gratefully acknowledge the support of the Tyrolean Science Fund (TWF). Sergiy Pereverzyev Jr.
gratefully acknowledges the support of the Austrian Science Fund (FWF): project P 29514-N32.
He also would like to thank Alessandro Verri, Vera Kurkova, Linh Nguyen, Jürgen Frikel,
Xin Guo,
Ding-Xuan Zhou, and members of Ding-Xuan Zhou's group at the City University of Hong Kong for
discussions concerning this work.



\begin{thebibliography}{10}

\bibitem{AgrFinKuc09}
M.~Agranovsky, D.~Finch, and P.~Kuchment.
\newblock Range conditions for a spherical mean transform.
\newblock {\em Inverse Probl. Imaging}, 3(3):373--383, 2009.

\bibitem{AlvRosLaw12}
M.~A. Alvarez, L.~Rosasco, and N.~D. Lawrence.
\newblock Kernels for vector-valued functions: A review.
\newblock {\em Found. Trends Mach. Learn.}, 4(3):195--266, 2012.

\bibitem{AmbKuc06}
G.~Ambartsoumian and P.~Kuchment.
\newblock A range description for the planar circular {R}adon transform.
\newblock {\em SIAM J. Math. Anal.}, 38(2):681--692, 2006.

\bibitem{BarFriNgu15}
L.~L. Barannyk, J.~Frikel, and L.~V. Nguyen.
\newblock On artifacts in limited data spherical {Radon} transform: curved
  observation surface.
\newblock {\em Inverse Probl.}, 32(1):015012, 2015.

\bibitem{Bea11}
P.~Beard.
\newblock Biomedical photoacoustic imaging.
\newblock {\em Interf. Focus}, 1(4):602--631, 2011.

\bibitem{BBG07}
P.~Burgholzer, J.~Bauer-Marschallinger, H.~{Gr\"un}, M.~Haltmeier, and
  G.~Paltauf.
\newblock {Temporal back-projection algorithms for photoacoustic tomography
  with integrating line detectors}.
\newblock {\em {Inverse Probl.}}, 23(6):S65--S80, 2007.

\bibitem{BurMatHalPal07}
P.~Burgholzer, G.~J. Matt, M.~Haltmeier, and G.~Paltauf.
\newblock Exact and approximate imaging methods for photoacoustic tomography
  using an arbitrary detection surface.
\newblock {\em Phys. Rev. E}, 75(4):046706, 2007.

\bibitem{CouHil62}
R.~Courant and D.~Hilbert.
\newblock {\em Methods of Mathematical Physics. Volume 2}.
\newblock Wiley-Interscience, New York, 1962.

\bibitem{FinHalRak07}
D.~Finch, M.~Haltmeier, and Rakesh.
\newblock Inversion of spherical means and the wave equation in even
  dimensions.
\newblock {\em SIAM J. Appl. Math.}, 68(2):392--412, 2007.

\bibitem{FinPatRak04}
D.~Finch, S.~Patch, and Rakesh.
\newblock Determining a function from its mean values over a family of spheres.
\newblock {\em SIAM J. Math. Anal.}, 35(5):1213--1240, 2004.

\bibitem{FinRak06}
D.~Finch and Rakesh.
\newblock The range of the spherical mean value operator for functions
  supported in a ball.
\newblock {\em Inverse Probl.}, 22(3):923--938, 2006.

\bibitem{FinRak09}
D.~Finch and Rakesh.
\newblock Recovering a function from its spherical mean values in two and three
  dimensions.
\newblock In L.~V. Wang, editor, {\em Photoacoustic imaging and spectroscopy},
  chapter~7, pages 77--88. CRC Press, 2009.

\bibitem{FriQui13}
J.~{Frikel} and E.~T. {Quinto}.
\newblock {Characterization and reduction of artifacts in limited angle
  tomography.}
\newblock {\em {Inverse Probl.}}, 29(12):21, 2013.

\bibitem{FriQui15}
J.~Frikel and E.~T. Quinto.
\newblock Artifacts in incomplete data tomography with applications to
  photoacoustic tomography and sonar.
\newblock {\em SIAM J. Appl. Math.}, 75(2):703--725, 2015.

\bibitem{GruJBO10}
H.~Gr\"un, T.~Berer, P.~Burgholzer, R.~Nuster, and G.~Paltauf.
\newblock Three-dimensional photoacoustic imaging using fiber-based line
  detectors.
\newblock {\em J. Biomed. Optics}, 15(2):021306, 2010.

\bibitem{Hal09}
M.~Haltmeier.
\newblock Frequency domain reconstruction for photo- and thermo\-acoustic
  tomography with line detectors.
\newblock {\em Math. Mod. Meth. Appl. Sci.}, 19(2):283--306, 2009.

\bibitem{Hal13}
M.~{Haltmeier}.
\newblock {Inversion of circular means and the wave equation on convex planar
  domains}.
\newblock {\em {Comput. Math. Appl.}}, 65(7):1025--1036, 2013.

\bibitem{Hal14}
M.~Haltmeier.
\newblock Universal inversion formulas for recovering a function from spherical
  means.
\newblock {\em {SIAM J. Math. Anal.}}, 41(1):214--232, 2014.

\bibitem{HalNgu17}
M.~Haltmeier and L.~V. Nguyen.
\newblock Analysis of iterative methods in photoacoustic tomography with
  variable sound speed.
\newblock {\em SIAM J. Imaging Sci.}, 10(2):751--781, 2017.

\bibitem{HalPer15}
M.~Haltmeier and S.~{Pereverzyev Jr}.
\newblock {Recovering a function from circular means or wave data on the
  boundary of parabolic domains}.
\newblock {\em {SIAM J. Imaging Sci.}}, 8(1):592--610, 2015.

\bibitem{HalPer15b}
M.~Haltmeier and S.~{Pereverzyev Jr.}
\newblock The universal back-projection formula for spherical means and the
  wave equation on certain quadric hypersurfaces.
\newblock {\em J. Math. Anal. Appl.}, 429(1):366--382, 2015.

\bibitem{HasTibFri09}
T.~Hastie, R.~Tibshirani, and J.~Friedman.
\newblock {\em {The Elements of Statistical Learning}}.
\newblock Springer Series in Statistics, 2009.

\bibitem{Her09}
G.~T. Herman.
\newblock {\em Fundamentals of computerized tomography: image reconstruction
  from projections}.
\newblock Springer, 2009.

\bibitem{HriKucNgu08}
Y.~Hristova, P.~Kuchment, and L.~Nguyen.
\newblock Reconstruction and time reversal in thermoacoustic tomography in
  acoustically homogeneous and inhomogeneous media.
\newblock {\em Inverse Probl.}, 24(5):055006 (25pp), 2008.

\bibitem{HuaWanNWA13}
C.~Huang, K.~Wang, L.~Nie, L.~V. Wang, and M.~A. Anastasio.
\newblock Full-wave iterative image reconstruction in photoacoustic tomography
  with acoustically inhomogeneous media.
\newblock {\em {IEEE} Trans. Med. Imaging}, 32(6):1097--1110, 2013.

\bibitem{KucKun11}
P.~Kuchment and L.~Kunyansky.
\newblock Mathematics of photoacoustic and thermoacoustic tomography.
\newblock In {\em Handbook of Mathematical Methods in Imaging}, pages 817--865.
  Springer, 2011.

\bibitem{KucKun08}
P.~Kuchment and L.~A. Kunyansky.
\newblock Mathematics of thermoacoustic and photoacoustic tomography.
\newblock {\em Eur. J. Appl. Math.}, 19:191--224, 2008.

\bibitem{Kun07}
L.~A. Kunyansky.
\newblock {Explicit inversion formulae for the spherical mean Radon transform}.
\newblock {\em {Inverse Probl.}}, 23(1):373--383, 2007.

\bibitem{Kun07b}
L.~A. Kunyansky.
\newblock A series solution and a fast algorithm for the inversion of the
  spherical mean {R}adon transform.
\newblock {\em Inverse Probl.}, 23(6):S11--S20, 2007.

\bibitem{Kun11}
L.~A. Kunyansky.
\newblock {Reconstruction of a function from its spherical (circular) means
  with the centers lying on the surface of certain polygons and polyhedra}.
\newblock {\em {Inverse Probl.}}, 27(2):025012, 2011.

\bibitem{LiWan09}
C.~Li and L.~V. Wang.
\newblock Photoacoustic tomography and sensing in biomedicine.
\newblock {\em Phys. Med. Biol.}, 54(19):R59, 2009.

\bibitem{MatLew96}
S.~Matej and R.~M. Lewitt.
\newblock Practical considerations for {3-D} image reconstruction using
  spherically symmetric volume elements.
\newblock {\em IEEE Trans. Med. Imag.}, 15(1):68--78, 1996.

\bibitem{MicPon05}
C.~A. Micchelli and M.~Pontil.
\newblock On learning vector-valued functions.
\newblock {\em Neural Comput.}, 17(1):177--204, 2005.

\bibitem{Nat12}
F.~Natterer.
\newblock {Photo-acoustic inversion in convex domains}.
\newblock {\em {Inverse Probl. Imaging}}, 6(2):315--320, 2012.

\bibitem{Lin14}
L.~V. Nguyen.
\newblock On a reconstruction formula for spherical {R}adon transform: {A}
  microlocal analytic point of view.
\newblock {\em Anal. Math. Phys.}, 4(3):199--220, 2014.

\bibitem{Ngu15}
L.~V. Nguyen.
\newblock On artifacts in limited data spherical {Radon} transform: flat
  observation surfaces.
\newblock {\em SIAM J. Math. Anal.}, 47(4):2984--3004, 2015.

\bibitem{PNH07}
G.~Paltauf, R.~Nuster, M.~Haltmeier, and P.~Burgholzer.
\newblock {Experimental evaluation of reconstruction algorithms for limited
  view photoacoustic tomography with line detectors}.
\newblock {\em {Inverse Probl.}}, 23(6):S81--S94, 2007.

\bibitem{PalNusHalBur07}
G.~Paltauf, R.~Nuster, M.~Haltmeier, and P.~Burgholzer.
\newblock Photoacoustic tomography using a {M}ach-{Z}ehnder interferometer as
  an acoustic line detector.
\newblock {\em App. Opt.}, 46(16):3352--3358, 2007.

\bibitem{PalViaPJ02}
G.~Paltauf, J.~A. Viator, S.~A. Prahl, and S.~L. Jacques.
\newblock Iterative reconstruction algorithm for optoacoustic imaging.
\newblock {\em J. Acoust. Soc. Am.}, 112(4):1536--1544, 2002.

\bibitem{Pat04}
S.~K. Patch.
\newblock Thermoacoustic tomography --- consistency conditions and the partial
  scan problem.
\newblock {\em Phys. Med. Biol.}, 49:2305--2315, 2004.

\bibitem{Pat09}
S.~K. Patch.
\newblock Photoacoustic and thermoacoustic tomography: Consistency conditions
  and the partial scan problem.
\newblock In {\em Photoacoustic Imaging and Spectroscopy}, pages 103--116. CRC
  Press, 2009.

\bibitem{RosNtzRaz13}
A.~Rosenthal, V.~Ntziachristos, and D.~Razansky.
\newblock Acoustic inversion in optoacoustic tomography: {A} review.
\newblock {\em Curr. Med. Imaging Rev.}, 9(4):318--336, 2013.

\bibitem{SchPerHal17}
J.~Schwab, S.~{Pereverzyev Jr.}, and M.~Haltmeier.
\newblock {A Galerkin least squares approach for photoacoustic tomography}.
\newblock Technical report, {University of Innsbruck, Department of
  Mathematics, Applied Mathematics Group}, 2017.
\newblock Preprint Nr. 30, arXiv:1612.08094 [math.NA].

\bibitem{SteUhl09}
P.~Stefanov and G.~Uhlmann.
\newblock Thermoacoustic tomography with variable sound speed.
\newblock {\em Inverse Probl.}, 25(7):075011, 2009.

\bibitem{SteUhl13}
P.~Stefanov and G.~Uhlmann.
\newblock Is a curved flight path in {SAR} better than a straight one?
\newblock {\em SIAM J. Appl. Math.}, 73(4):1596--1612, 2013.

\bibitem{WanSchSOA14}
K.~Wang, R.~W. Schoonover, R.~Su, A.~Oraevsky, and M.~A. Anastasio.
\newblock Discrete imaging models for three-dimensional optoacoustic tomography
  using radially symmetric expansion functions.
\newblock {\em IEEE Trans. Med. Imag.}, 33(5):1180--1193, 2014.

\bibitem{XiaYaoWan14}
J.~Xia, J.~Yao, and L.~V. Wang.
\newblock Photoacoustic tomography: principles and advances.
\newblock {\em {Prog. Electromagnetics Res.}}, 147:1--22, 2014.

\bibitem{XuWan05}
M.~Xu and L.~V. Wang.
\newblock Universal back-projection algorithm for photoacoustic computed
  tomography.
\newblock {\em Phys. Rev. E}, 71(1):0167061--0167067, 2005.

\bibitem{XuWan06}
M.~Xu and L.~V. Wang.
\newblock Photoacoustic imaging in biomedicine.
\newblock {\em Rev. Sci. Instruments}, 77(4):041101 (22pp), 2006.

\bibitem{XuWanAmbKuc04}
Y.~Xu, L.~V. Wang, G.~Ambartsoumian, and P.~Kuchment.
\newblock Reconstructions in limited-view thermoacoustic tomography.
\newblock {\em Med. Phys.}, 31(4):724--733, 2004.

\bibitem{XuXuWan02}
Y.~Xu, M.~Xu, and L.~V. Wang.
\newblock Exact frequency-domain reconstruction for thermoacoustic
  tomography--{II}: Cylindrical geometry.
\newblock {\em IEEE Trans. Med. Imag.}, 21:829--833, 2002.

\bibitem{YaoJia11}
L.~Yao and H.~Jiang.
\newblock Photoacoustic image reconstruction from few-detector and
  limited-angle data.
\newblock {\em Biomed. Opt. Express}, 2(9):2649--2654, 2011.

\bibitem{ZanSchHal09b}
G.~Zangerl, O.~Scherzer, and M.~Haltmeier.
\newblock Exact series reconstruction in photoacoustic tomography with circular
  integrating detectors.
\newblock {\em Commun. Math. Sci.}, 7(3):665--678, 2009.

\end{thebibliography}

\end{document}